\documentclass[12pt]{amsart}
\usepackage{amsmath,amsfonts,amsthm,amscd,stmaryrd, amssymb,mathrsfs}
\usepackage[all,cmtip]{xy}
\usepackage{enumitem}
\usepackage{hyperref}
\usepackage{tikz, tikz-cd}
\newtheorem{theorem}{Theorem}[section]

\newtheorem{proposition}[theorem]{Proposition}
\newtheorem{lemma}[theorem]{Lemma}
\newtheorem{corollary}[theorem]{Corollary}
\theoremstyle{definition}
\newtheorem{definition}[theorem]{Definition}
\newtheorem{remark}[theorem]{Remark}

\newtheorem{assumption}[theorem]{Assumption}

\newcommand{\RR}{\mathbb{R}}

\newcommand{\Coker}{{\rm{Coker}}}

\renewcommand{\hat}{\widehat}
\renewcommand{\tilde}{\widetilde}

\DeclareMathOperator{\Ker}{Ker}

\usepackage[letterpaper]{geometry}
\numberwithin{equation}{section}
\begin{document}
\title[Collapsing CSC metrics]{Collapsing constant scalar curvature metrics}
\author{Jeff Viaclovsky}\address{Department of Mathematics, University of California, Irvine, USA}\email{jviaclov@uci.edu}
\author{Xiaokang Wang}\address{Department of Mathematics, University of California, Irvine, USA}\email{xiaokangw@uci.edu}
\date{July 29, 2026}
\begin{abstract} 
We prove that a sequence of constant scalar curvature (CSC) metrics which is collapsing with bounded curvature to a manifold $(X,g_{\infty})$ can be perturbed to a sequence of $\mathcal{N}$-invariant collapsing CSC metrics, under a natural assumption involving the eigenvalues of the drift Laplacian on $(X,g_{\infty})$.  This answers a special case of a question of Cheeger-Fukaya-Gromov. We also give some natural conditions on the limiting metric-measure space under which the eigenvalue assumption is automatically satisfied. 
\end{abstract}
\maketitle

\section{Introduction}
The setting for this paper is a sequence $\{(M^n_i ,g_i)\}$ of closed $n$-dimensional Riemannian manifolds
which limit in the Gromov-Hausdorff sense to a limit space $(X,g_{\infty})$.  If $\dim(X) = n$, the sequence is called \textit{non-collapsing},
but if $\dim(X) = d < n$, the sequence is said to be \textit{collapsing}. In this case, when the sequence $g_i$ has uniformly bounded sectional curvature, the definitive theory is the Cheeger-Fukaya-Gromov (CFG) theory \cite{CFG}, which was developed through the foundational works of Cheeger, Fukaya, and Gromov in the late 1980s and early 1990s. In this paper, we focus on the case when the limiting metric space is a lower dimensional manifold.
\begin{assumption}\label{assumption: collapsing CSC intro}
    Let $\{(M^n_i ,g_i)\}$ be a sequence of closed $n$-dimensional Riemannian manifolds such that:
\begin{itemize}
    \item The Riemannian curvature is uniformly bounded: \begin{align}
        \Vert\operatorname{Rm}(g_i)\Vert_{C^{0}} \leq C_0,
\end{align}
    \item The diameter is bounded from above: $d(M_i, g_i)\leq C$.
    \item The sequence collapses to a lower-dimensional manifold, i.e.:
    \begin{align}
        (M^n_i,g_i)\to (X^d, g_\infty)
    \end{align}
    in the Gromov-Hausdorff sense to a manifold $X$, where $d<n$. 

\end{itemize}
\end{assumption}
We will let $\tau_i$ denote the Gromov-Hausdorff distance of $ (M^n_i,g_i)$ and $(X^d, g_\infty)$.
Under these assumptions, Cheeger-Fukaya-Gromov proved that there is an $\mathcal{N}$-structure, defined as follows.
\begin{definition}
    Let $M^n$, $X^d$ be two smooth manifolds. Let $F:M\to X$ be a smooth fibration. We say $F$ is an \textit{infranil-fibration} if for each fiber over $x\in X$, $F^{-1}(x)$ is diffeomorphic to an infranilmanifold (See Definition~\ref{d:infranil}). We say $M$ admits a \textit{$\mathcal{N}$-structure} if there is a cover of $X$, $\{U_j\}$, where $U_j$ is simply connected so that the local universal cover $\tilde{F^{-1}(U_j)}$ admits a nilpotent $\mathcal{N}$-action such that on overlaps the actions are conjugate by an affine isomorphism of $\mathcal{N}$ (see Definition~\ref{d:nil_action}).
\end{definition}
Furthermore, Cheeger-Fukaya-Gromov proved the sequence can be perturbed to an $\mathcal{N}$-invariant sequence $(M_i, \hat{g}_i)$ (see Definition~\ref{d:invariant} below). In this paper, we give a detailed analysis of the Cheeger-Fukaya-Gromov construction under the above assumptions. The original proof of Cheeger-Fukaya-Gromov involves a Ricci-flow smoothing step, but this step changes the limit space. We incorporate some modern theory (see for example \cite{NaberZhang,SunZhang}), and prove the following finite regularity estimate on the invariant metric.
\begin{theorem}\label{t.invariant_nearby_metric_estimate_intro}
Let $\{(M_i,g_i)\}$ and $(X,g_\infty)$ satisfy the Assumption~\ref{assumption: collapsing CSC intro}. 
Then, for $i$ large enough, there exists an infranil-fibration 
$F_i: M_i \rightarrow X$ and a $\mathcal{N}$-invariant metric $\hat{g}_i$ on $M_i$ so that, for every $U\subseteq X$, the pull-back metric on $\tilde{F_i^{-1}(U)}$, still denoted as $\hat{g}_i$ and $g_i$, satisfies the estimate:
\begin{align}
    \Vert \hat{g}_i - g_i\Vert_{C^{0,\alpha}}\leq C(n,C_0,X)\cdot\tau_i
\end{align}
for any $0<\alpha<1$, and
\begin{align}
    \Vert \hat{g}_i - g_i\Vert_{W^{1,p}_{loc}(K)}\leq C(n,C_{0,p},X,K)\cdot \tau_i
\end{align}
for $p > n$, and for any compact set $K\subseteq \tilde{F_i^{-1}(U)}$. 
\end{theorem}
We prove this Theorem in Section~\ref{s:cfg} and Section~\ref{s:estinv} essentially by tracing through every step used in Cheeger-Fukaya-Gromov's proof in \cite{CFG}, which also uses a crucial result from Ruh in \cite{ruh1982almost}.
We will use this estimate below to answer a special case of question posed by Cheeger-Fukaya-Gromov in the constant scalar curvature (CSC) case. 
Their question was: if the original sequence has some special geometric structure (such as Einstein, K\"ahler, etc.), can one find the invariant sequence having the same geometric structure? (See \cite{Huang-Rong-Wang, SunZhang} for some recent progress of CFG question in other cases.)
We find that the answer to their question in the CSC case is yes in the case that the limit is a manifold,  provided that a natural assumption on the eigenvalues of the drift Laplacian on the limit space is satisfied. To understand this condition,
recall that the limit space has a natural metric measure structure, inherited from the collapsing geometry. For this measured limit space 
$(X^d, g_{\infty}, \mu)$, the measure $\mu$ is absolutely continuous with respect to the Riemannian measure,
so there exists $\chi \in C^{1,\alpha}(X) \cap W^{2,p}(X)$ with $d\mu =  \chi dvol(g_{\infty})$. 
Our condition involves the drift Laplacian
\begin{align}
\label{d:drift}
\Delta_\mu h \equiv \Delta_X h  + \nabla \log \chi \cdot \nabla_Xh,
\end{align} 
which is the natural self-adjoint Laplacian operator with respect to $d\mu$. Our main result is the following. 
\begin{theorem}
\label{cfgconj_intro}
 Let $\{(M_i^n, g_i)\}$ be a sequence of n-dimensional closed Riemannian manifolds satisfying Assumption~\ref{assumption: collapsing CSC intro}.
Assume in addition that 
\begin{enumerate}
    \item The scalar curvature $S(g_i)\equiv v$, where $v$ is a constant;
    \item If $v > 0$,  $\frac{v}{n-1} \notin \operatorname{Spec}( -\Delta_\mu)$. 
\end{enumerate}
Then, for $i$ large, there exists smooth metrics $g'_i$ on $M_i$ such that:
\begin{itemize}
    \item $g_i$, $g_i'$ are close in $W^{1,p}$ sense, for any $p>n$, with $\Vert g_i-g_i' \Vert_{W^{1,p}}\to 0$ ;
    \item $g_i'$ is $\mathcal{N}$-invariant;
    \item $S(g_i')$ has constant scalar curvature satisfying $S(g_i') \equiv  v$  if $v \neq 0$, but $S(g_i') \equiv v_i$ if  $v = 0$, where $v_i \to 0$ as $i \to \infty$.
    \item $(M_i, g_i')\to (X, g_\infty)$ in the Gromov-Hausdorff sense, as $i \to \infty$.
\end{itemize}
\end{theorem}
The proof of this theorem will be given in Section~\ref{section: main}, and is fundamentally a perturbation argument using the implicit function theorem. More precisely, we show that for $i$ sufficiently large, Cheeger-Fukaya-Gromov's invariant metric can be conformally deformed to constant scalar curvature by an invariant conformal factor; see Section~\ref{s:conformal perturbation}.   Note that the main difficulty in the proof is that the invariant metric does not necessarily have curvature tensor close to the original $g_i$, since the CSC equation does not give higher regularity for the metric. Nevertheless, we show in Section~\ref{s.scalar distribution} that using the Sobolev regularity, the scalar curvatures are close in a distributional norm. This will then be sufficient to apply the implicit function theorem to an operator mapping into a space of distributions. 

\begin{remark}
As mentioned above, the standard trick to obtain the higher regularity is a Ricci flow smoothing technique. While it might be possible to prove Theorem~\ref{cfgconj_intro} with such an approach, since this method changes the limit space, it would seem that some finite regularity considerations would still be necessary, so such an approach would be basically equivalent to our direct method. 
\end{remark}

\begin{remark} Our proof should extend without much difficulty if the limit is assumed to be a Riemannian orbifold. In \cite{CFG}, CFG considered the much more general situation of a singular collapsed limit space. It is possible that our proof could be extended to the general case, but would require more tools from analysis on singular spaces, so we will not consider it here. 
\end{remark}

We next make some remarks about the eigenvalue assumption in Theorem~\ref{cfgconj_intro}. 
In the case $v \leq 0$, there is no condition. Note that when $v = 0$, we cannot control the sign of the $v_i$, this seems to be unavoidable
since $0 \in \operatorname{Spec}( -\Delta_\mu)$ in this case. 
We let $\chi = e^{-f}$, then in terms of $f$, the drift Laplacian is 
\begin{align}
\Delta_{\mu} = \Delta_{g_{\infty}} - \langle \nabla f, \nabla \cdot \rangle.
\end{align}
We show below in Section~\ref{sec:convergence_eigenvalue} that under Assumption~\ref{assumption: collapsing CSC intro}, that $f \in W^{2,p}(X)$ for $p > n$; see Theorem~\ref{t:chiconv}.
In the case $v > 0$, we recall the following weighted Lichnerowicz-Obata estimate: if
    \begin{align}
    \label{e:ricnfintro}
 Ric(g_{\infty}) + \nabla^2 f - \frac{1}{n-d} df \otimes df \geq \frac{v}{n} g_{\infty},
\end{align}
almost everywhere, then
\begin{align}
\label{e:lamb1}
    \lambda_1( -\Delta_{\mu}) > \frac{v}{n-1};
\end{align}
see Subsection~\ref{s:eignevalue estimate} for further discussion and references. We note here the following corollary, which requires slight additional regularity of $f$.
\begin{corollary}
\label{c:lambda1} 
Let $f \in C^2(X)$. If $v > 0$, 
\begin{align}
\label{Riclower}
    Ric(g_{\infty}) \geq \Big( \frac{v}{n} + \eta \Big) g_{\infty},
\end{align}
where $\eta > 0$ is a constant, and 
\begin{align}
    \Vert \nabla^2 f \Vert_{C^0} + \frac{1}{n-d} \Vert \nabla f \Vert_{C^0}^2 \leq \eta, 
\end{align}
then 
\begin{align}
\label{e:lambda12}
     \lambda_1(- \Delta_{\mu}) > \frac{v}{n-1}.
\end{align}
\end{corollary}
In other words, the corollary says that if the Ricci curvature of the limit space is bounded below as in \eqref{Riclower},   and if $f$ is sufficiently close to a constant in $C^2$-norm (with an effective estimate), then the eigenvalue condition \eqref{e:lambda12} is satisfied.

\subsection{Acknowledgements}  The first author was partially supported by NSF Grant DMS-2404195. The authors are indebted to Ruobing Zhang for numerous helpful discussions about collapsing theory, and to John Lott for discussions and references regarding the Lichnerowicz-Obata estimate. 


\section{Preliminaries}
\label{s:pre}
In this Section, we recall some basic notions. 
\subsection{Function space}
We will use the following function spaces. Let $M$ be a smooth compact manifold. Choose a finite atlas of coordinate charts $\{(U_i,\phi_i)\}_{i=1}^N$. Let $\{\rho_i\}_{i=1}^N$ be a smooth partition of unity subordinate to the cover $\{U_i\}$. Thus, for any function $u: M\to\mathbb{R}$, we can consider its coordinate representation: 
\begin{align}
    u_i = (\rho_i u ) \circ\phi_i^{-1}:\mathbb{R}^n \to \mathbb{R}.
\end{align}
\begin{definition}
    A function $u:M\to\mathbb{R}$ has $C^{k,\alpha}$ \textit{regularity}, for $k\geq 0$, $0\leq \alpha\leq 1$, if the norm:
    \begin{align}
        \Vert u\Vert_{C^{k,\alpha}(M)} = \sum_{i=1}^N \Vert u_i\Vert_{C^{k,\alpha}(\mathbb{R}^n)}
    \end{align}
    is finite.
    A function $u:M\to\mathbb{R}$ has $W^{k,p}$ \textit{regularity}, for $k\in\mathbb{N}$, $p\geq 1$, if the norm:
    \begin{align}
        \Vert u\Vert_{W^{k,p}(M)} = \sum_{i=1}^N \Vert u_i\Vert_{W^{k,p}(\mathbb{R}^n)}
    \end{align}
    is finite.   Let $f:M\to N$ be a map from manifold $M$ to $N$. Then $f$ has $C^{k,\alpha}$ ($W^{k,p}$) \textit{regularity}, if we choose coordinate charts $\{(V_i,\psi_i)\}_{i=1}^N$ so that $\psi_i\circ f: M\to \mathbb{R}$ has $C^{k,\alpha}$ ($W^{k,p}$) regularity, for each $i$.
\end{definition}

We also need to consider the convergence of functions on a sequence of manifolds. For this, we need to introduce Cheeger-Gromov convergence.
\begin{definition}
    Let $\{M_i\}$ be a sequence of compact manifolds. Suppose there exists a compact manifold $M_\infty$ so that there exists a sequence of diffeomorphisms
    \begin{align}
        \Phi_i: M_\infty \to M_i.
    \end{align}
    Then we say a sequence of functions $u_i:M_i\to \mathbb{R}$ convergent in $C^{k,\alpha}$ ($W^{k,p}$) \textit{Cheeger-Gromov sense} if $\{u_i\circ\Phi_i\}$ converge on $M_\infty$ in $C^{k,\alpha}$ ($W^{k,p}$) sense.
\end{definition}
For a sequence of collapsing manifolds, we have the following sense of convergence. 
\begin{definition}\label{d.local_convergence}
 Let $\{(M_i^n,g_i)\}$ be a sequence of compact $n$-dimensional smooth manifolds which converges to $(X^d,g_\infty)$, a compact $d$-dimensional smooth manifold in the Gromov-Hausdorff sense. We say that the convergence is in \textit{local $C^{k,\alpha}$ $(W_{loc}^{k,p})$ sense} if there is a uniform constant $r>0$ so that for every sequence $(B_{x_i}(r),g_i)\to (B_{x_\infty}(r),g_\infty)$, we have the following diagram
    \begin{center}
    \begin{tikzcd}
    & (\tilde{B}_{\tilde{x}_i}(r), \tilde{g}_i, G_i) \arrow[r,  "C^{k,\alpha}"] \arrow[d, "\pi_i"]
    & (\tilde{B}_{\tilde{x}_\infty}(r),\Tilde{g}_\infty, G_\infty) \arrow[d, swap, "\pi_\infty"]\\
    & (B_{x_i}(r), g_i) \arrow[r, "\text{GH}"]
    & (B_{x_\infty}(r), g_\infty)
    \end{tikzcd},
    \end{center}
    with the upper sequence converging in the $C^{k,\alpha}$ $(W_{loc}^{k,p})$ Cheeger-Gromov sense.
\end{definition}

\subsection{Nilpotent group and actions}
We now recall some basic facts about nilpotent Lie groups and their actions.
\begin{definition}[Nilpotent group]
    A \textit{nilpotent Lie group} is a Lie group, $\mathcal{N}$, whose Lie algebra is $\mathfrak{g}$ is a nilpotent Lie algebra, that is 
    \begin{align}
        \mathfrak{g}_1 = [\mathfrak{g},\mathfrak{g}],\quad \mathfrak{g}_2 = [\mathfrak{g}, \mathfrak{g}_1], ...
    \end{align}
    eventually vanishes, i.e., $\mathfrak{g}_{\ell} = 0$, for some finite $\ell >0$.
\end{definition}

\begin{remark} Any connected nilpotent Lie group $\mathcal{N}$ is diffeomorphic to $\mathbb{R}^n$, given by the exponential map. This can be seen from the Baker-Campbell-Hausdorff formula since it terminates in finitely many steps. 
\end{remark}

\begin{definition}[Automorphism and Left action]
    The \textit{automorphism group} of a nilpotent Lie group, \(\operatorname{Aut}(\mathcal N)\) is the group of Lie group automorphisms of \(\mathcal N\):
\begin{align}
\operatorname{Aut}(\mathcal N)
=
\{\phi:\mathcal N\to \mathcal N:
\phi \text{ is a smooth group isomorphism}\}.
\end{align}

Let $h\in\mathcal{N}$, then we define the left action:
\begin{align}
    L_h: \mathcal{N} \to \mathcal{N}, \quad x\mapsto hx.
\end{align}
Then, we denote the group of \textit{left translations}:
\begin{align}
    \mathcal{N}_L : = \{L_h:h\in \mathcal{N}\}.
\end{align}
\end{definition}
It can be easily seen that there is a natural isomorphism $\mathcal{N}_L\cong \mathcal{N}$.
Similarly, using 
\begin{align}
    R_h: \mathcal{N} \to \mathcal{N}, \quad x\mapsto xh,
\end{align}
we can define the \textit{right translation} group $\mathcal{N}_R\cong\mathcal{N}$.
\begin{definition}[Left and right invariant vector fields]
    For \(V\in\mathfrak n\), the corresponding \textit{left-invariant vector field} is
\begin{align}
V^L_g=(L_g)_*V,
\end{align}
and the corresponding \textit{right-invariant vector field} is
\begin{align}
V^R_g=(R_g)_*V.
\end{align}
\end{definition}
It can be shown that left action is generated by a right invariant vector field
\begin{align}
\widetilde V_g
=
\left.\frac{d}{dt}\right|_{t=0}\exp(tV)g
=
(R_g)_*V
=
V^R_g.
\end{align}

There is a natural flat affine connection on \(\mathcal N\), which we define next.
\begin{definition}[Affine connection and affine structure]
    The \textit{canonical affine connection}, $\nabla^{can}$, is the connection for which the left-invariant vector fields are parallel, i.e.,
\begin{align}
\nabla^{\mathrm{can}}V^L=0
\end{align}
for every left-invariant vector field \(V^L\).
\end{definition}
Note that, $\nabla^{\mathrm{can}}$ is flat, and, in general, not torsion free, in fact:
\begin{align}
    T^{\nabla^{\mathrm{can}}}(X,Y) = -[X,Y].
\end{align}
Since the left-invariant vector fields are parallel, so is $T^{\nabla^{can}}$.

Note also that, when $\nabla^{\mathrm{can}}$ acts on a right invariant vector field, we have:
\begin{align}\label{eq:right invariant vf}
    \nabla^{\mathrm{can}}_{W^R}V^R = [V,W]^R,
\end{align}
which is a constant right invariant vector field.
\begin{definition}\label{d.affine structure}
An \textit{affine structure} on a smooth manifold is given by a choice of a flat, affine connection $\nabla$ on $TM$ with parallel torsion.
     
     Two affine manifolds, $(M,\nabla^M)$, $(N,\nabla^N)$ are \textit{affine equivalent}, if there is a diffeomorphism $\phi: M\to N$, so that $\phi^*\nabla^N = \nabla^M$.

     The \textit{affine group}, $\operatorname{Aff}(M,\nabla^M)$ (or just $\operatorname{Aff}(M)$), of $(M,\nabla^M)$, is defined to be the affine isomorphism group of $(M,\nabla^M)$.
\end{definition}
Note that $\operatorname{Aff}(\mathcal{N})$ of a nilpotent Lie group $\mathcal{N}$ is equivalent to the group that preserves the left invariant vector fields. Thus,
\begin{align}
    \operatorname{Aff}(\mathcal{N}) = \mathcal{N}_L\rtimes \operatorname{Aut}(\mathcal{N}).
\end{align}

Now, we define nilmanifolds and infranilmanifolds. 

\begin{definition}\label{d:infranil}
    A compact \textit{nilmanifold} is a quotient
\begin{align}
M=\Lambda\backslash\mathcal N,
\end{align}
where \(\Lambda\subset\mathcal{N}_L\) is a discrete cocompact subgroup.

A compact \textit{infranilmanifold} is a quotient
\begin{align}
Z=\Gamma\backslash\mathcal N,
\end{align}
where \(\Gamma\) is a torsion-free discrete cocompact subgroup of $\mathcal{N}_L\rtimes \operatorname{Aut}(\mathcal{N})$.
\end{definition}
It can be shown that $\Lambda = \Gamma\cap \mathcal{N}_L$, the pure translation subgroup, is a normal subgroup of $\Gamma$ with finite index. Thus, an infranilmanifold is finitely covered by a nilmanifold. Furthermore, the canonical affine connection descends to an affine connection on the quotient, so that the group $\operatorname{Aff}(\Gamma\backslash \mathcal{N})$ is defined. 

Finally, we will talk about the structure of the affine group of an infranilmanifold, denoted by $\operatorname{Aff}(\Gamma\backslash \mathcal{N})$, and its identity component $\operatorname{Aff}^0(\Gamma\backslash \mathcal{N})$.
\begin{proposition}
\label{p:aff}
Let $\Gamma\subseteq \mathcal{N}_L\rtimes \operatorname{Aut}(\mathcal{N})$ be a cocompact lattice group, then we have
\begin{align}
    \operatorname{Aff}(\Gamma\backslash \mathcal{N}) \cong N_{\operatorname{Aff}(\mathcal{N})}(\Gamma)/\Gamma
\end{align}
and
\begin{align}
    \operatorname{Aff}^0(\Gamma\backslash \mathcal{N}) \cong C_{\operatorname{Aff}(\mathcal{N})}(\Gamma)/\Gamma,
\end{align}
where $N_{\operatorname{Aff}(\mathcal{N})}(\Gamma)$ and $C_{\operatorname{Aff}(\mathcal{N})}(\Gamma)$ are normalizer and centralizer of $\Gamma$ in $\operatorname{Aff}(\mathcal{N})$, respectively.
\end{proposition}
\begin{proof}
    For every element in $\operatorname{Aff}(\Gamma\backslash \mathcal{N})$ can be lifted to $\operatorname{Aff}( \mathcal{N})$ that makes every $\Gamma$-orbit to a $\Gamma$-orbit. Thus, the lift of $\operatorname{Aff}(\Gamma\backslash \mathcal{N})$ can be identified with $N_{\operatorname{Aff}(\mathcal{N})}(\Gamma)$. Thus, we have $\operatorname{Aff}(\Gamma\backslash \mathcal{N}) \cong N_{\operatorname{Aff}(\mathcal{N})}(\Gamma)/\Gamma$ after the quotient.

    Note that since $\Gamma$ is a discrete subgroup, and the identity map is in the centralizer. Thus, we have $\operatorname{Aff}^0(\Gamma\backslash \mathcal{N}) \cong C_{\operatorname{Aff}(\mathcal{N})}(\Gamma)/\Gamma$.
\end{proof}
\begin{remark} 
\label{r:NR} 
It is natural to identify the centralizer of the left action group $\mathcal{N}_L$ as the right action subgroup $\mathcal{N}_R$. Indeed, for any $g\in \mathcal{N}_L$ and $\phi\in C_{\operatorname{Aff(\mathcal{N})}}(\mathcal{N}_L)$, if $\phi(g(x)) = g\phi(x)$, for any $x\in\mathcal{N}$, then if we choose $x = e$, let $\phi(e) = h$, then $\phi(g) = gh$. Thus, we may view $C_{\operatorname{Aff}(\mathcal{N})}(\Gamma)$ as a subgroup of $\mathcal{N}_R$, which acts trivially on the right-invariant vector fields.
\end{remark}

\section{Cheeger-Fukaya-Gromov fibration theorem}
\label{s:cfg}

Let $\{(M^n_i ,g_i)\}$ be a sequence of closed $n$-dimensional Riemannian manifolds such that:
\begin{itemize}
    \item The Riemannian curvature is uniformly bounded: $|\operatorname{Rm}(g_i)|_{g_i} \leq C_1$;
    \item The diameter is bounded from above: $d(M_i, g_i)\leq C_2$.
\end{itemize}
Then from the general theory, we know that:
\begin{align}
    (M_i^n ,g_i) \to (X,d)
\end{align}
In the Gromov-Hausdorff sense, to some metric space $(X,d)$. In general, by Theorem 0.6 in \cite{Fuk88}, the dimension of $X$ is an integer. 

From now on, we assume in our case $(X,d)$ is a compact Riemannian manifold of dimension $d<n$, which we 
denote by $(X^d, g_\infty)$.  

 We recall the fibration theorem of Cheeger-Fukaya-Gromov. The fibration theorem was first obtained by \cite{Fuk87}, then improved by \cite{CFG}. It was improved further by Naber-Zhang and Sun-Zhang latter. We next state the following version which includes the finite regularity. 
 
\begin{theorem}{\cite{Fuk87}, \cite{CFG}, \cite{NaberZhang},  \cite{SunZhang}}
\label{t.cfg} 
Let $\{(M^n_i,g_i)\}$, $(X^d,g_\infty)$ be defined as above, for $i$ sufficiently large, for any $0 < \alpha < 1$, there exists a smooth fiber bundle map:
\begin{align}
    F_i : (M^n_i ,g_i)\to (X^d, g_\infty)
\end{align}
with a sequence $\tau_i \to 0$ such that:
\begin{itemize}
    \item $F_i$ is a $\tau_i$-Gromov-Hausdorff approximation;
    \item There exists $C$ such that there exists a uniform $C^{2,\alpha}$ estimate: $\Vert F_i \Vert_{C^{2,\alpha}} < C$;
    \item The second fundamental form of the fiber:
    \begin{align}
    \label{e:sff}
        \Vert A_{F_i^{-1}(x)} \Vert_{C^{0,\alpha}} \leq C_0,
    \end{align}
    for some uniform constant $C_0(n, C_1)$;
    \item $F_i$ is almost-Riemannian submersion, i.e. $\forall v$, vector that is orthonormal to the fiber of $F_i$, we have:
    \begin{align}
        (1-\tau_i)|v|_{g_i}\leq |dF_i(v)|_{g\infty}\leq (1+\tau_i)|v|_{g_i}.
    \end{align}
\end{itemize}
    
\end{theorem}
The proof can be found in the Appendix. In our application, we focus mainly on finite H\"older and Sobolev regularity. We next state the following finite regularity version of the theorem with a Riemannian curvature $C^k$ bound assumption (although we only require the case $k=0$ in this paper).

\begin{corollary}\label{c.fibration_regularity}
    Let $\{(M^n_i,g_i)\}$, $(X^d,g_\infty)$ be defined as in Theorem~\ref{t.cfg}.  In addition, we assume
        \begin{align}
            \Vert\operatorname{Rm}(g_i)\Vert_{C^{k}} \leq C_k, 
            \end{align}
            for some $k \geq 0$, then for $i$ large the fibration map satisfies
        \begin{align}
            \Vert F_i\Vert_{C^{k+2,\alpha}} \leq C(n,C_k, X),
        \end{align}
        the second fundamental form of the fiber satisfies
        \begin{align}
            \Vert A_{F_i^{-1}(x)}\Vert_{C^{k,\alpha}}\leq C(n,C_k,X),
        \end{align}
        for any $0<\alpha<1$. We also have
        \begin{align}
           \Vert F_i\Vert_{W^{k+3,p}} \leq C(n,C_{k},X) 
        \end{align}
        and the second fundamental form of the fiber satisfies
        \begin{align}
            \Vert A_{F_i^{-1}(x)}\Vert_{W^{k+1,p}}\leq C(n,C_{k},X),
        \end{align}
for any $p > 1$.
\end{corollary}
\begin{proof} The proof just needs some additional remarks to the argument in the Appendix. 
It follows from the standard convergence theory. The construction of the fiberation is that Step 1 we construct a local fiberation using harmonic coordinate, then in Step 2 we construct cut-off functions to patch them together. 
    
We now focus on the regularity issue, so that, around each $x_i\in M_i$, the local universal covers converge in $C^{k+1,\alpha}$ or in $W^{k+1,p}_{loc}$ in the sense of Definition~\ref{d.local_convergence}; see, for example, \cite{grove1997comparison}.

For the H\"older estimate, note that we have local $C^{k+1,\alpha}$ estimate for the metrics, then the harmonic function constructed in Step 1 satisfies uniform $C^{k+2,\alpha}$ estimates, by standard Schauder estimates. Also, the cut-off function constructed in Step 2 has uniform $C^{k,\alpha}$ estimates, so we are done.
  
In the Sobolev case, note that we have local $W^{k+2,p}$ estimates for the metrics. To get the $W^{k+3,p}$ estimate on $F_i$, notice that on functions:
\begin{align}
    \Delta(\nabla u ) = \nabla\Delta u + \operatorname{Ric}(\nabla u).
\end{align}
Thus, we apply the $W^{k+2,p}$ estimate on this equation and get the desired estimate for the harmonic function constructed in Step 1 and the cut-off function constructed in Step 2, and the result follows.  
\end{proof}

Now we need to discuss the structure of the fiber. From \eqref{e:sff}, the second fundamental form of any fiber is uniformly bounded, so we can conclude that all the fibers of $F_i$ are almost-flat manifolds:
    \begin{align}
        \operatorname{diam}(F_i^{-1}(x_i))^2\cdot |\operatorname{Rm}_{F_i^{-1}(x_i)}| < \tau(\epsilon),
    \end{align}
    for $i$ sufficiently large. For a single fiber, we recall the following theorem of Ruh, to which we add several explicit finite regularity estimates.

\begin{theorem}[\cite{ruh1982almost}]\label{t.ruh}
    There is a small constant, $\epsilon(n)>0$ so that the following holds. Let $(Z,h)$ be a smooth manifold so that:
    \begin{enumerate}
        \item[(1)] The Riemannian curvature 
        \begin{align}
            \Vert\operatorname{Rm}(h)\Vert_{C^{k}} \leq C_k,
        \end{align}
        for some $k\geq 0$. 
        \item[(2)] The diameter of $(Z,h)$, $\delta: = \operatorname{diam}(Z,h)$:
        \begin{align}
            \delta\leq \epsilon(n).
        \end{align}
    \end{enumerate}
    Then we have $Z\simeq \Gamma\backslash \mathcal{N}$, where $\mathcal{N}$ is a simply-connected nilpotent Lie group, and $\Gamma$ is a cocompact subgroup of $\mathcal{N}_L\rtimes \operatorname{Aut}(\mathcal{N})$, where $\mathcal{N}_L \cong \mathcal{N}$ is the left translation subgroup, i.e., $Z$ is an infranilmanifold. Also, $\Lambda = \mathcal{N}\cap \Gamma$ is normal in $\Gamma$ with $\#(\Lambda\backslash\Gamma)\leq w_0$ for some constant $w_0 = w_0(n)$. Moreover, for every point $z\in Z$, there is a flat connection $\nabla^z$ on the orthonormal frame bundle $FZ$ of $Z$ satisfying the following:
    \begin{enumerate}
        \item[(1)] For every $q,q'\in Z$, there exists a gauge transformation, $g^{q,q'}: FZ \rightarrow FZ$ so that $g^{q,q'}\nabla^{q'} = \nabla^q$, with the estimate
        \begin{align}
        \label{gqc0}
            |g^{q,q'}-\operatorname{Id}|\leq C(n)\cdot\delta,
        \end{align}
        and 
        \begin{align}
            \Vert g^{q,q'} - \operatorname{Id} \Vert_{C^{k+1,\alpha}} \leq C(n,k)\cdot \delta
        \end{align}
        for any $0<\alpha<1$, or
        \begin{align}
            \Vert g^{q,q'} - \operatorname{Id} \Vert_{W^{k+2,p}}\leq C(n,k,p)\cdot\delta
        \end{align}
        for any $p > n$;
        \item[(2)] The difference of the associated connection $\nabla^z$ with the Levi-Civita connection $\nabla^{LC}$ on $TZ$:
        \begin{align}
            \Vert \nabla^{LC}-\nabla^z\Vert_{C^{k,\alpha}} \leq C(n,k)\cdot \delta
        \end{align}
        for any $0<\alpha<1$, or
        \begin{align}
            \Vert \nabla^{LC}-\nabla^z\Vert_{W^{k+1,p}}\leq C(n,k,p)\cdot\delta
        \end{align}
        for any $p > n$.
    \end{enumerate}
    \end{theorem}
\begin{proof}
    The topological classification is proved in \cite[Section~2]{ruh1982almost}. The proof relies on E. Ruh's construction of flat connections on almost flat manifolds, combined with a standard geometric rescaling argument and elliptic regularity theory.   Next, we summarize Ruh's approach in \cite{ruh1982almost}, so that we may consider the finite regularity estimates. 

    \textbf{Step 1: Rescaling the Metric.} \\
     Let $\tilde{h} = \delta^{-2} h$. Then the diameter of the new Riemannian manifold $(Z, \tilde{h})$ is exactly $1$. The Riemannian curvature tensor and its covariant derivatives scale accordingly. Specifically, the norm of the curvature under the rescaled metric satisfies
    \begin{align}
        \Vert \operatorname{Rm}(\tilde{h}) \Vert_{C^0(\tilde{h})} = \delta^2 \Vert \operatorname{Rm}(h) \Vert_{C^0(h)} \leq C_0 \delta^2.
    \end{align}
    Similar for the higher derivative bounds. Because $\delta \leq \epsilon(n)$ is chosen to be sufficiently small, the rescaled manifold $(Z, \tilde{h})$ has bounded geometry, diameter $1$, and arbitrarily small curvature.

    \textbf{Step 2: Ruh's Construction and Elliptic Regularity.} \\
    Following Ruh's approach, for any $z \in Z$, Ruh constructs a flat connection with parallel torsion $\tilde{\nabla}^z$ on the frame bundle, which is a perturbation of the Levi-Civita connection: $\tilde{\nabla}^z = \tilde{\nabla}^{LC} - \tilde{\eta}^z$, where $\tilde{\eta}^z$ is a connection $1$-form taking values in $\mathfrak{o}(n)$, which is orthogonal to the space of harmonic $1$-forms. 
    
    To fix the gauge, $\tilde{\eta}^z$ is constructed to be co-closed with respect to $\tilde{\nabla}^{LC}$ (i.e., $d^* \tilde{\eta}^z = 0$). The condition that $\tilde{\nabla}^z$ is flat ($F_{\tilde{\nabla}^z} = 0$) gives the Maurer-Cartan type equation:
    \begin{align}
    \label{Ruheqn}
        d\tilde{\eta}^z + [\tilde{\eta}^z, \tilde{\eta}^z] = \operatorname{Rm}(\tilde{h}).
    \end{align}
    Ruh constructs the solution $\tilde{\eta}^z$ of  \eqref{Ruheqn} which has parallel torsion and which satisfies 
    \begin{align}
    \label{Ruhest}
    \Vert \tilde{\eta}^z \Vert_{W^{1,p}} \leq C(n,p)\delta^2 \end{align} 
    for any $p > n$, and is orthogonal to the space of harmonic $1$-forms. 
    By taking the exterior co-derivative $d^*$, $\tilde{\eta}^z$ satisfies a nonlinear elliptic PDE of the form
    \begin{align}
        \Delta \tilde{\eta}^z = d^* \Big(\operatorname{Rm}(\tilde{h}) -  [\tilde{\eta}^z, \tilde{\eta}^z] \Big) ,
    \end{align}
    where $\Delta = dd^* + d^*d$ is the Hodge Laplacian.  Because of ellipticity, we can apply standard Schauder and Calderón-Zygmund elliptic regularity estimates on the rescaled manifold $(Z, \tilde{h})$. Specifically, if $\operatorname{Rm}(\tilde{h}) \in W^{k,p}$, elliptic theory implies $\tilde{\eta}^z \in W^{k+1,p}$. Furthermore, the elliptic estimates yield:
    \begin{align}
        \Vert \tilde{\eta}^z \Vert_{W^{1,p}(\tilde{h})} \leq \tilde{C}(n,p) \Vert \operatorname{Rm}(\tilde{h})  - [\tilde{\eta}^z, \tilde{\eta}^z] \Vert_{L^p(\tilde{h})} \leq C'(n,k,p) \delta^2,
    \end{align}
for $\tilde{\eta}^z$ orthogonal to the space of harmonic $1$-forms, and where we used \eqref{Ruhest} and Sobolev embedding to estimate the nonlinear term. 
    Scaling back to the original metric $h$, the connection $1$-form scales as $\eta^z = \delta \tilde{\eta}^z$. Therefore, the difference $\nabla^{LC} - \nabla^z = \eta^z$ satisfies the bound:
    \begin{align}
        \Vert \nabla^{LC} - \nabla^z \Vert_{W^{1,p}(h)} \leq C(n,k,p) \cdot \delta.
    \end{align}
A similar argument gives $W^{k+1,p}$-estimate for $k \geq 1$. The Schauder estimates follow similarly.

    \textbf{Step 3: Regularity of the Gauge Transformation.}   \\
    For any $q, q' \in Z$, let $\nabla^q$ and $\nabla^{q'}$ be the flat connections constructed above. The gauge transformation $\mathfrak{G} = g^{q,q'} \in \Gamma(\operatorname{Aut}(TZ))$ constructed in \cite{ruh1982almost} satisfies $\mathfrak{G} \nabla^{q'} = \nabla^q \mathfrak{G}$.
    
    Writing the connections in terms of their $1$-forms $\eta^q$ and $\eta^{q'}$ relative to the Levi-Civita connection, it is equivalent to the first-order linear PDE:
    \begin{align}
        &\mathfrak{G}^{-1}d\mathfrak{G} - \mathfrak{G}^{-1}\eta^{q'}\mathfrak{G} = -\eta^q\\
        &d\mathfrak{G} = \eta^{q'} \mathfrak{G} - \mathfrak{G} \eta^q.
    \end{align}
  We now find the optimal regularity of $g$. From Step 2, we established that $\eta^q, \eta^{q'} \in W^{1,p}$ for $p > n$.
      Using \eqref{gqc0}, $d\mathfrak{G}$ gives the full gradient, therefore, the first derivative $d\mathfrak{G}$ is in $W^{1,p}$, which concludes that $\mathfrak{G} \in W^{2,p}$. Thus:
    \begin{align}
        \Vert \mathfrak{G} - \operatorname{Id} \Vert_{W^{2,p}} \leq C \left( \Vert \eta^{q'} \Vert_{W^{1,p}} + \Vert \eta^q \Vert_{W^{1,p}} \right) \leq C(n,p) \cdot \delta.
    \end{align}
    Higher Sobolev regularity follows similarly. By replacing the Sobolev spaces with Hölder spaces, the exact same bootstrap argument on $d\mathfrak{G}= \eta^{q'} \mathfrak{G} - \mathfrak{G} \eta^q$ demonstrates that if $\eta \in C^{k,\alpha}$, then $\mathfrak{G} \in C^{k+1,\alpha}$ with bounds scaling linearly with $\delta$. 
\end{proof}
From the Gauss equations on a fiber $F_i^{-1}(x)$, we have 
\begin{align}
\Vert Rm(F_i^{-1}(x)) \Vert_{C^k} \leq C \Big(\Vert Rm(g_i) \Vert_{C_k}  + \Vert A_{F_i^{-1}(x)} * A_{F_i^{-1}(x)} \Vert_{C^k} \Big)
\end{align}
From Theorem~\ref{t.cfg}, Corollary~\ref{c.fibration_regularity}, and Theorem~\ref{t.ruh}, the conclusions of Theorem~\ref{t.ruh} hold on each fiber.
Next, we study the finite regularity of these fiber affine structures in the base variables, based on the results of \cite{CFG}. 
\begin{theorem}[\cite{CFG}]\label{t.CFG_connection_regularity}
Let $\{(M^n_i,g_i)\}$, $(X^d,g_\infty)$ be defined as in Theorem~\ref{t.cfg}. Assume that:
\begin{align}
        \Vert\operatorname{Rm}(g_i)\Vert_{C^{k}} \leq C_k,
\end{align}
for some $k\geq 0$.
Let $U\subseteq X$ be a coordinate neighborhood on $X$. Then, for $i$ sufficiently large, there exists a family of flat connections, $\nabla^x_i$, $x\in U$, on the fiber, $Z^x_i = F_i^{-1}(x)$, $x\in U$ with the following properties:
\begin{enumerate}
    \item[(1)] $\nabla^x_i$ is uniformly $C^{k+1,\alpha}$ or $W^{k+2,p}$ only in the base variable $x$;
    \item[(2)] For each fiber $Z^x_i$, the difference of the flat connection and the induced Levi-Civita  connection $\nabla^{LC}_i$ on the fiber has the following estimate:
    \begin{align}
        \Vert \nabla^x_i -\nabla^{LC}_i\Vert_{C^{k,\alpha}}\leq C(n, C_k, X)\cdot \tau_i
    \end{align}
    for any $0<\alpha<1$, or
    \begin{align}
        \Vert \nabla^x_i -\nabla^{LC}_i\Vert_{W^{k+1,p}}\leq C(n, C_{k,p}, X)\cdot \tau_i
    \end{align}
    for any $p > n$;
      \item[(3)] There exists $\Gamma$, $\mathcal{N}$ (not depending on $x$), such that for each fiber $Z^x_i$, there exists an affine isomorphism (cf. Definition~\ref{d.affine structure})
    \begin{align}
        \psi_x : \Gamma\backslash \mathcal{N} \to Z^x_i
    \end{align}
    with uniform $C^{k+2,\alpha}$ or $W^{k+3,p}$ estimate, respectively. Moreover, it is uniformly $C^{k+1,\alpha}$ or $W^{k+2,p}$ in base variable $x$.
\end{enumerate}

\end{theorem}

\begin{proof}
Parts (1) and (2) follow from the Cheeger-Fukaya-Gromov construction of the local invariant pure nilpotent structures \cite[Section~4]{CFG}. The bounded curvature condition $\Vert\operatorname{Rm}(g_i)\Vert_{C^{k}} \leq C_k$ ensures that, locally, the geometry is controlled in $C^{k+1,\alpha}$ or $W^{k+2,p}$. The affine connection $\nabla^x_i$ is constructed as in \cite{CFG} by an averaging process of the construction in Theorem~\ref{t.ruh} in the $z$-variable, which involves the integration of the $g^{q,q'} \nabla^{q'}$ against the invariant measure. It has $C^{k+1,\alpha}$ or $W^{k+2,p}$ regularity, since it only depends on the underlying Riemannian structure. The difference tensor $S_i = \nabla^x_i - \nabla_i^{LC}$ satisfies the elliptic estimates bounded by the scale of the collapse $\tau_i$, yielding the uniform bounds in $C^{k,\alpha}$ (or $W^{k+1,p}$) on each fiber, which gives Part~(2). Since we are averaging with respect to the Riemannian volume form, the connection $\nabla^x_i$ itself has $C^{k+1,\alpha}$ (or $W^{k+2,p}$) regularity in the parameter $x$, which gives us Part~(1). 

For part (3), we must construct the affine isomorphism $\psi_x : \Gamma\backslash \mathcal{N} \to Z^x_i$ and establish its regularity.
First, we need to apply Malcev's Theorem. For any simply connected open subset $U\subseteq X$, by choosing a local trivialization, consider $F_i^{-1}(U)\simeq U\times \Gamma\backslash\mathcal{N}$ and $\tilde{F_i^{-1}(U)}\simeq U\times\mathcal{N}$. This induces a family of homomorphism from $\pi_1(Z_i^x)$ to $\Gamma$. As in \cite{CFG}, By Malcev's Theorem, any finitely generated torsion-free nilpotent group $\Gamma$ uniquely embeds as a uniform (cocompact) lattice in a unique simply connected nilpotent Lie group $\mathcal{N}$. The real form of $\mathcal{N}$ is uniquely determined by the rational completion of $\Gamma$. This gives us our standard model space $\Gamma\backslash\mathcal{N}$. 

As in \cite{CFG}, the connection $\nabla^x_i$ on $Z^x_i$ constructed in part (1) is flat and has parallel torsion. Therefore, $Z^x_i$ is an affinely flat manifold. We consider the universal cover $\tilde{Z}^x_i$, equipped with the pullback connection $\tilde{\nabla}^x_i$. Because $\tilde{\nabla}^x_i$ is flat and has parallel torsion, we can construct a developing map $D_x : \tilde{Z}^x_i \to \mathcal{N}$. By the \cite{CFG}, the holonomy group of $\nabla^x_i$ is trivial, and the monodromy representation maps $\pi_1(Z^x_i) = \Gamma$ into the affine group $\mathcal{N}_L \rtimes \operatorname{Aut}(\mathcal{N})$. The developing map is an equivariant affine diffeomorphism between $\tilde{Z}^x_i$ and $\mathcal{N}$
hence this map descends to the quotient by the action of $\Gamma$
\begin{align}
    \psi_x : \Gamma\backslash \mathcal{N} \to Z^x_i.
\end{align}

The regularity of the map $\psi_x$ is determined by the regularity of the affine connection $\nabla^x_i$ from which the developing map $D_x$ is constructed. The developing map is essentially obtained by fixing a frame and integrating the parallel transport equation. In particular, it solves
\begin{align}
    \frac{\partial^2 (\psi_x)^a}{\partial u^b \partial u^c} + (\Gamma_i^x)_{bc}^k \frac{\partial (\psi_x)^a}{\partial u^k} = 0
\end{align}
where $u$ denotes the local coordinates on $\Gamma\backslash\mathcal{N}$. It immediately follows that if the connection $\nabla^x_i$ is in $C^{k,\alpha}$, then $\psi_x$ is uniformly bounded in $C^{k+2,\alpha}$ with respect to the fiber variables.
Similarly, if $\nabla^x_i \in W^{k+1,p}$, the resulting map $\psi_x$ is in $W^{k+3,p}$.

Finally, the dependence of $\psi_x$ on the base point $x \in U$ is determined by the variation of the connections $\nabla^x_i$ with respect to $x$, which is of $C^{k+1,\alpha}$ (or $W^{k+2,p}$).
\end{proof}

\section{Estimates on the invariant metrics}
\label{s:estinv}

In this section, we discuss nilpotent group actions, give estimates on the invariant metrics, discuss the scalar curvature as a distribution, and also discuss some properties of limiting measure and the drift Laplacian on the base. 

\subsection{Nilpotent group action and invariant metric}
Not only is there a fibration structure, there is a certain type of group action on the collapsing sequence, which we define next.

\begin{definition}[Nilpotent Group action]\label{d:nil_action}
    Let $U\subseteq X$ be a simply connected neighborhood of $x\in X$. A \textit{local Nilpotent Group action (local $\mathcal{N}$-action)} on $\tilde{F_i^{-1}(U)}$ is defined by choosing a local trivialization: $\Phi: F_i^{-1}(U)\to U\times \Gamma\backslash\mathcal{N}$, which lifts to the universal cover to $\tilde{\Phi}: \tilde{F_i^{-1}(U)} \rightarrow  U\times \mathcal{N}$, on which the action is the pull-back of the action generated by the right invariant vector fields on $\mathcal{N}$. A \textit{global Nilpotent Group action ($\mathcal{N}$-action)} on $M_i$ is a covering of $X$, $\{U_{j}\}$, where $U_j$ is simply connected,
    together with local $\mathcal{N}$-actions on $\tilde{F_i^{-1}(U_{j})}$ such that on overlaps the actions are conjugate by an affine isomorphism of $\mathcal{N}$.
\end{definition}
Note that, in general, the existence of local nilpotent fibration does not guarantee the existence of the nilpotent group action. We need to show that first the local nilpotent group action exists and then we need to show the transition map will reduce to the affine isomorphism so that the local group action will be preserved.
\begin{theorem}[\cite{CFG}]\label{t.CFG_existence_group_action}
Under the conditions of Theorem~\ref{t.cfg}, the structure group of the fibration reduced to $\operatorname{Aff}^0(\Gamma\backslash\mathcal{N})\cong C_{\operatorname{Aff}(\mathcal{N})}(\Lambda)/\Lambda\subset \operatorname{Aff}(\Gamma\backslash \mathcal{N})$. Thus there exists a global $\mathcal{N}$-action on $M_i$ for $i$ sufficiently large.
\end{theorem}
\begin{proof}
We first show that for any simply connected open set $U\subseteq X$, there exists a Nilpotent group action on $\tilde{F_i^{-1}(U)}$, for all $i$ sufficiently large, which is independent of the choice of the local trivialization. 
We first choose any local trivialization
$\Phi: F_i^{-1}(U)\to U\times \Gamma\backslash\mathcal{N}$, which lifts to the universal cover to $\tilde{\Phi}: \tilde{F_i^{-1}(U)} \rightarrow  U\times \mathcal{N}$. Fix $x\in U$. By the proof of Theorem~\ref{t.CFG_connection_regularity}, for each $y\in U$, the isomorphism of the first fundamental group 
of $Z_i^y$ with $Z_i^x$ induces a unique affine isomorphism
\begin{align}
    \tilde{\Phi}_y : \tilde{Z_i^x} \to\tilde{Z_i^y}.
\end{align} 
As argued in Section 4 of \cite{CFG}, any right invariant vector field on $\mathcal{N}$ is then viewed as a vector field on $\tilde{Z_i^x}$ that
 can be extended to $\tilde{F_i^{-1}(U)}$ by pushing forward to each $y$ fiber using $\tilde{\Phi}_y$. Thus, we have defined a $\mathcal{N}$-action on $\tilde{F_i^{-1}(U)}$ using $\Phi$.

As in the Section 4 of \cite{CFG}, the $\mathcal{N}$-action defined above is independent of local trivialization, since the local trivialization is defined unique up to a section of affine transformations, $\sigma: U \rightarrow \operatorname{Aff}^0(\Gamma \backslash \mathcal{N}) \cong C_{\operatorname{Aff}(\mathcal{N})}(\Gamma)/\Gamma$ (from Proposition~\ref{p:aff}).
Since the latter group is a subgroup of $N_R$, it acts trivially on right invariant vector fields on each fiber; see Remark~\ref{r:NR}. Also, on the overlap of $U$ and $U'$, we can apply Malcev's Theorem and the argument in the proof of Theorem~\ref{t.CFG_connection_regularity} and the above argument to conclude the existence of a well-defined global $\mathcal{N}$-action.
\end{proof}
By \textit{right-invariant vector field} we mean the pull-back of a constant right invariant vector on $\mathcal{N}$ under a local trivialization. This is independent of the local trivialization, so gives a well-defined global vector field on $M_i$.
Now, we will continue our finite regularity discussion. We will continue to assume the sequence satisfies the following assumptions.
\begin{assumption}\label{assumption: sequence assumption}
    Let $\{(M^n_i ,g_i)\}$ be a sequence of closed $n$-dimensional Riemannian manifolds such that:
\begin{itemize}
    \item The Riemannian curvature is uniformly bounded: \begin{align}
        \Vert\operatorname{Rm}(g_i)\Vert_{C^{k}} \leq C_k,
\end{align}
for some $k\geq 0$, 
    \item The diameter is bounded from above: $d(M_i, g_i)\leq C$.
\end{itemize}
\end{assumption}
We can prove the following finite regularity lemma for local trivializations.
\begin{lemma}\label{l.local_trivialization_regularity}
    Let $\{(M_i,g_i)\}$ and $(X,g_\infty)$ satisfy the Assumption~\ref{assumption: sequence assumption}. Then, for each small simply connected neighborhood $U\subseteq X$, there exists a local trivialization $\Phi_{U,i}: F_i^{-1}(U)\to U\times \Gamma\backslash \mathcal{N}$ with uniform $C^{k+2,\alpha}$, for any $0 < \alpha < 1$ and $W^{k+3,p}$ estimate for $p > n$.
\end{lemma}
\begin{proof}
    With the assumption above, by Corollary~\ref{c.fibration_regularity}, the fibration map $F_i:M_i\to X$, satisfies uniform $\Vert F_i\Vert_{C^{k+2,\alpha}}$ or $\Vert F_i\Vert_{W^{k+3,p}}$ estimate. Let $U\subseteq X$ be a simply connected open subset. Since $F_i$ are submersions, by the implicit function theorem, there exists a local section $\sigma_i:U\to F_i^{-1}(U)\subseteq M_i $, so that $F_i\circ\sigma_i = \operatorname{Id}$ on $U$ with uniform $\Vert\sigma_i\Vert_{C^{k+2,\alpha}}$ or $\Vert\sigma_i\Vert_{W^{k+3,p}}$ estimate respectively.

    On the local universal cover, $\tilde{F_i^{-1}(U)}$, we have the fiber-wise affine isomorphism:
    \begin{align}
        \tilde{\psi}_x: \mathcal{N} \to \tilde{Z_i^x}.
    \end{align}
    With a little abuse of notation, we denote $\tilde{\psi}_x:\mathcal{N}_L \to \operatorname{Aut}(\tilde{Z_i^x})$ as the isomorphism of group action.
    By Theorem~\ref{t.CFG_existence_group_action}, the fiber-wise group action generated by the right invariant vector fields on $\mathcal{N}$ extends to a group action on $\tilde{F_i^{-1}(U)}$. For each fiber $\tilde{Z_i^{x'}}$, we have the map
    \begin{align}
        \phi_{x'}: \operatorname{Aut}(Z_i^x) \to \operatorname{Aut}(Z_i^{x'}).
    \end{align}
    Thus, we define the fiber-wise action on $x'$ by: $\psi_{x'}: = \phi_{x'}^{-1}\circ \tilde{\psi_x}$.
    which gives us the action as 
    \begin{align}
        \tilde{\psi}_i:\mathcal{N}_L &\to \operatorname{Aut}(\tilde{F_i^{-1}(U))}
    \end{align}
 with the property that $F_i(\tilde{\psi}_i(g)p) = F_i(p)$, for any $p\in M_i$, $g\in\mathcal{N}_L$. By (3) of Theorem~\ref{t.CFG_connection_regularity} and Theorem~\ref{t.CFG_existence_group_action}, the mapping $\tilde{\psi}_i$ has a uniform $C^{k+2,\alpha}$ or $W_{loc}^{k+3,p}$ estimate, respectively. 

 Thus, we can define our local trivialization
 \begin{align}
     \tilde{\Phi_{U,i}^{-1}}: U\times \mathcal{N} &\to \tilde{F_i^{-1}(U)}\\
     (x,g)&\mapsto \tilde{\psi}_i(g)\cdot\tilde{\sigma_i}(x),
 \end{align}
where $\tilde{\sigma_i}$ is a lift to $\tilde{F_i^{-1}(U)}$. Note that, by Theorem~\ref{t.CFG_existence_group_action}, $\tilde{\psi}_i$ is invariant under $\Gamma$. Thus, $\tilde{\Phi_{U,i}^{-1}}$ descends to $U\times \Gamma\backslash \mathcal{N}$, which is the desired local trivialization with regularity.
\end{proof}
\begin{remark}
    The construction of this local trivialization is different from the one from \cite{CFG}. In \cite{CFG}, they use the normal exponential map to construct a tubular neighborhood of the center fiber. In our case, we use the improved regularity given by the harmonic coordinate, which is optimal. Also, instead of using the family of affine isomorphism of the fiber using Riemannian structure, we directly use the fiber-wise isomorphism given by Malcev's Theorem, which does not lose regularity.
\end{remark}

Since the definition of right invariant vector field is independent of the choice of the local trivialization, we have the following regularity for right invariant vector fields.
\begin{corollary}
\label{c:rivfwr}
Let $\{(M_i,g_i)\}$ and $(X,g_\infty)$ satisfy the Assumption~\ref{assumption: sequence assumption}. Let $V$ be a right-invariant vector field on $\mathfrak{B} = \tilde{F_i^{-1}(U)}\cap B_{\tilde{x_i}}(2\tau_i)$. Then,
\begin{align}
\label{rivfwr1}
    \Vert V\Vert_{C^{k+1,\alpha}(\mathfrak{B})}\leq C(n,C_k,X)\cdot |V(\tilde{x}_i)|,
\end{align}
for any $0<\alpha<1$, or
\begin{align}
\label{rivfwr2}
    \Vert V\Vert_{W^{k+2,p}(\mathfrak{B})}\leq C(n,C_{k,p},X)\cdot |V(\tilde{x}_i)|,
\end{align}
for any $p > n$.
\end{corollary}
\begin{proof}
    Note, the right-invariant vector field on $\tilde{F_i^{-1}(U)}$ is given by the pull-back of the right-invariant vector field using $\tilde{\Phi_{U,i}}$. Then, the $C^0$ ($L^\infty$) norm of $V$ in both cases is uniformly bounded by $|V(\tilde{x}_i)|$, depending on the $C^0$ norm of the Riemannian metric. 
    
    Let $S^x:=\nabla_i^x - \nabla_i^{LC}$ on the fiber. Denote $A^x$ as the second fundamental form of the fiber of $F_i^{-1}(x)$, which we regard as a tensor on the universal cover $\tilde{F_i^{-1}(U)}$. By Theorem~\ref{t.cfg}, $\Vert A^x\Vert_{C^0}$ is uniformly bounded. Since $V$ is a right invariant vector field, by \eqref{eq:right invariant vf}, we have:
    \begin{align}
        \Vert \nabla V\Vert_{C^0} \leq\Vert S^xV\Vert_{C^0} +\Vert A^xV\Vert_{C^0} + \Vert \nabla^{x}_iV\Vert_{C^0} \leq C\Vert V\Vert_{C^0}.
    \end{align}
    By ODE comparison, we have the desired $C^{1}$ estimate. For the higher order estimate, since $\nabla^{x}_iV$ produces a constant right invariant vector field (recall \eqref{eq:right invariant vf}), we have
    \begin{align}
        \Vert \nabla^k V\Vert_{C^0} \leq\Vert (\nabla^{k-1}S^x)V\Vert_{C^0} +\Vert (\nabla^{k-1}A^x)V\Vert_{C^0}.
    \end{align}
   A similar estimate holds for the H\"older semi-norms. Thus, the estimate \eqref{rivfwr1} follows from Corollary~\ref{c.fibration_regularity} and Theorem~\ref{t.CFG_connection_regularity}. The proof of \eqref{rivfwr2} is similar; the details are omitted. 
\end{proof}

Now, we are ready to construct the $\mathcal{N}$-invariant metric using an averaging method of \cite{CFG}. For this, we first define what it means for a tensor to be $\mathcal{N}$-invariant.
\begin{definition}
\label{d:invariant}
    The tensor $T_i$ on $M_i$ is called \textit{$\mathcal{N}$-invariant} if for every $x\in X$, there is a simply connected neighborhood $U\subseteq X$ such that $F_i^{-1}(U)\simeq U\times \mathcal{N}/\Gamma$. Then on the universal cover $\tilde{F_i^{-1}(U)\simeq} U\times \mathcal{N}$, the pull-back tensor $\tilde{T}_i$ is invariant under $\mathcal{N}$-action.
\end{definition}
We next have the following result containing explicit finite regularity estimates on the invariant metric.
\begin{theorem}\label{t.invariant_nearby_metric_estimate}
Let $\{(M_i,g_i)\}$ and $(X,g_\infty)$ satisfy the Assumption~\ref{assumption: sequence assumption}. Then, for $i$ large enough, there exists a $\mathcal{N}$-invariant metric $\hat{g}_i$ on $M_i$ so that, for every $U\subseteq X$, the pull-back metric on $\tilde{F_i^{-1}(U)}$, still denoted as $\hat{g}_i$ and $g_i$, satisfies the estimate:
\begin{align}
    \Vert \hat{g}_i - g_i\Vert_{C^{k,\alpha}}\leq C(n,C_k,X)\cdot\tau_i
\end{align}
for any $0<\alpha<1$, and
\begin{align}
    \Vert \hat{g}_i - g_i\Vert_{W^{k+1,p}_{loc}(K)}\leq C(n,C_{k,p},X,K)\cdot \tau_i
\end{align}
for $p > n$, and for any compact set $K\subseteq \tilde{F_i^{-1}(U)}$.
\end{theorem}
\begin{proof}
    The construction follows the averaging method in the section 4 of \cite{CFG}. By Theorem~\ref{t.cfg}, $\Lambda = \mathcal{N}_L\cap \Gamma$ is normal in $\Gamma$ with uniform finite index, by passing to the finite cover, it is enough to consider the case $\Gamma\subseteq\mathcal{N}_L$. 
    
    On $\tilde{F_i^{-1}(U)}\simeq U\times \mathcal{N}$. Let $h\in \mathcal{N}_L$, then we regard $h\in \operatorname{Aut}(\tilde{F_i^{-1}(U)})$. Consider $h\mapsto h^*g_i$. Note that this map is constant on the $\Gamma$ orbit. Since the group $\mathcal{N}_L$ is nilpotent, it is unimodular. Therefore, the space $\Gamma\backslash \mathcal{N}_L$ admits a canonical left-invariant measure $d\mu$ of total volume 1. Thus, we construct the metric:
    \begin{align}
        \hat{g}_i = \int_{\Gamma\backslash\mathcal{N}} h^*g_i d\mu.
    \end{align}
    Thus, $\hat{g}_i$ is a left-invariant metric on $\tilde{F_i^{-1}(U)}$. 

    For each $h\in\mathcal{N}_L$, there exists a right invariant vector field $V_h$ so that $h = \operatorname{exp}(V_h)$. By the fundamental theorem of calculus, we have:
    \begin{align}
        h^*g_i - g_i = \int_0^1 \frac{d}{dt} \left(\operatorname{exp}^*(t V_h)g_i \right) dt = \int_0^1 (\operatorname{exp}^*(tV_h))(\mathcal{L}_{V_h}g_i) dt.
    \end{align}
    Thus, we have:
    \begin{align}
        \Vert \hat{g}_i - g_i\Vert \leq C_n\max_{h\in \Gamma\backslash\mathcal{N}}(\Vert D\operatorname{exp}(V_h)\Vert\cdot \Vert \nabla V_h\Vert),
    \end{align}
where the norm can be taken to be either the $C^{k,\alpha}$ norm or the above $W^{k+1,p}_{loc}$ norm. 
    Now by the collapsing assumption, the diameter of each fiber is less than $\tau_i$. For every $x\in U$, choose a lift $\tilde{x}_i$ on $\tilde{F_i^{-1}}(U)$. Then $|V_h(\tilde{x}_i)|\leq C\cdot \tau_i$. By the Jacobi field expansion (see, for example,\cite[Section~5]{LeeParker}), the $\ell$-th nilpotentcy of the group $\mathcal{N}$, and Theorem~\ref{t.CFG_connection_regularity}, 
    \begin{align}\Vert D\operatorname{exp}(V_h)\Vert = \Vert\sum_{n=0}^{\infty}\frac{(-1)^n}{(n+1)!}\operatorname{ad}_{V_h}^n \Vert \leq \Vert\sum_{n=0}^{\ell}\frac{(-1)^n}{(n+1)!}\operatorname{ad}_{V_h}^n \Vert\leq  \Vert(1+C\tau_i)\operatorname{Id}\Vert \leq C.
    \end{align}
    By Corollary~\ref{c:rivfwr}, $\Vert \nabla V_h\Vert\leq C\cdot \tau_i$, so we are done.
\end{proof}
\begin{remark} For $k = 0$, the curvature tensor of $\hat{g}_i$ is not necessarily close to the original $g_i$. This is because the averaging procedure of Cheeger-Fukaya-Gromov in the proof of Theorem~\ref{t.cfg} 
produces an affine connection which is close to Levi-Civita only in $C^{0,\alpha}$ or $W^{1,p}$ norms.
There might be some different way to do this averaging procedure which avoids this loss, but the authors are not aware of any method to improve this.
\end{remark}
As a corollary, we have the following uniform estimate Riemannian curvature of the invariant metric.
\begin{corollary}\label{c.invariant_metric_curvature_estimate}
With the Assumption~\ref{assumption: sequence assumption}. we have the following estimate on $\tilde{F_i^{-1}(U)}$:
\begin{enumerate}
    \item[(1)] In the H\"older case, if $k\geq 2$, then we have:
    \begin{align}
        \Vert \operatorname{Rm}(\hat{g}_i) - \operatorname{Rm}(g_i)\Vert_{C^{k-2,\alpha}} \leq C(n,C_k,X)\cdot\tau_i
    \end{align}
    for any $0<\alpha<1$.
    \item[(2)] In the Sobolev case, if $k\geq 1$ and $p > n$, then
    \begin{align}
        \Vert \operatorname{Rm}(\hat{g}_i) - \operatorname{Rm}(g_i)\Vert_{W^{k-1,p}_{loc}(K)} \leq C(n,C_{k,p},X,K)\cdot\tau_i
    \end{align}
    for any compact set $K\subseteq \tilde{F_i^{-1}(U)}$. 
    \end{enumerate}
\end{corollary}
\begin{remark}
    Note that, by using Lemma~\ref{l.local_trivialization_regularity}, the estimates in Theorem~\ref{t.invariant_nearby_metric_estimate} and Corollary~\ref{c.invariant_metric_curvature_estimate} can be made uniform on $U\times\mathcal{N}$.
\end{remark}

\subsection{Scalar curvature distribution}\label{s.scalar distribution}
For $k = 0$, the metric $\hat{g}_i$ only has $W^{1,p}$ bounds, thus its scalar curvature is not uniformly bounded. However, 
it turns out that the difference $S(\hat{g}_i) - S(g_i)$ \textit{is} bounded in a distributional norm. We explain this next. 

Recall Theorem~\ref{t.invariant_nearby_metric_estimate} and Corollary~\ref{c.invariant_metric_curvature_estimate}, if we only assume Riemannian curvature is uniformly bounded, i.e.
\begin{align}
    \Vert\operatorname{Rm} \Vert_{C^0}\leq C_0,
\end{align}
with $k =0$,  then we cannot conclude that 
the curvature of the $\mathcal{N}$-invariant metric is close to that of the original metric in the strong sense. However, for the scalar curvature, we have a notion of scalar curvature as a distribution, which involves negative Sobolev spaces. This is defined in Lee-LeFloch, which we next review; see~\cite{Lee_LeFloch}.

Consider two Riemannian metrics on $M$, $g$ and $h$. Let $\nabla^g$, $\nabla^h$ be their corresponding Levi-Civita connections. Define the following tensor of differences of connections $\Upsilon = \Gamma(g)-\Gamma(h)$, where in locall coordinate:
\begin{align}
    \Upsilon_{ij}^k := \frac{1}{2}g^{kl}(\nabla^h_i g_{jl} + \nabla^h_j g_{il} - \nabla^h_l g_{ij}).
\end{align}
By a straightforward computation, we have the following formula for the difference of two scalar curvature, $S(g)$ and $S(h)$.
\begin{proposition}[\cite{Lee_LeFloch}]\label{p:scalar_distribution}
    The difference of scalar curvature is given by this formula:
    \begin{align}
        S(g) - S(h) = \operatorname{div}_h T + Q,
    \end{align}
    where:
    \begin{align}
        T^k &= g^{ij}\Upsilon_{ij}^k - g^{ik}\Upsilon_{ji}^j\\
        Q &= - \nabla^h_k g^{ij}\Upsilon_{ij}^k + \nabla^h_k g^{ik}\Upsilon_{ji}^i + g^{ij}(\Upsilon_{kl}^k\Upsilon_{ij}^l - \Upsilon_{jl}^k\Upsilon_{ik}^l).
    \end{align}
\end{proposition}
\begin{proof}
    By computation:
    \begin{align}
    \begin{split}
    S(g) - S(h) &= g^{ij}(\nabla^h_k \Upsilon_{ij}^k - \nabla^h_j \Upsilon_{ki}^k) + g^{ij}(\Upsilon_{kl}^k\Upsilon_{ij}^l - \Upsilon_{jl}^k\Upsilon_{ik}^l)\\
    & = \nabla^h_k(g^{ij}\Upsilon_{ij}^k - g^{ik}\Upsilon_{ji}^j) - \nabla^h_k g^{ij}\Upsilon_{ij}^k + \nabla^h_k g^{ik}\Upsilon_{ji}^i + g^{ij}(\Upsilon_{kl}^k\Upsilon_{ij}^l - \Upsilon_{jl}^k\Upsilon_{ik}^l)\\
    & = \operatorname{div}_h T + Q.
\end{split}
\end{align}
\end{proof}
Thus, we can define $S(g) - S(h)$ as a distribution as follows.
\begin{definition}[\cite{Lee_LeFloch}]
    Let $M$ be a smooth manifold. Let $g$, $h$ be two smooth metrics so that $g,h\in W^{1,\infty}_{loc}(M)$. Then the difference of scalar curvature distribution is defined, for any $\phi\in W^{1,p}_{\mathrm{c}}(M)$, for some $p\geq 1$, by
    \begin{align}
        \langle S(g) - S(h), \phi \rangle = \int_M \Big(-T \cdot \nabla^h \Big(\phi \frac{dvol(g)}{dvol(h)}\Big) + Q\phi\frac{dvol(g)}{dvol(h)}\Big)dvol(h),
    \end{align}
    where $\frac{dvol(g)}{dvol(h)}$ is the Radon-Nikodym derivative of the two measures induced by $g$ and $h$.
\end{definition}
It is clear that, when $g$ and $h$ are at least $C^2$, 
\begin{align}
\langle S(g) - S(h), \phi \rangle = \int_M (S(g)- S(h))\phi dvol(g). 
\end{align}
\begin{proposition}\label{p:Scalar curvature estimate negative sobolev}
    If $g,h\in W^{1,\infty}_{loc}(M)$, then $\langle S(g) -S(h), \cdot\rangle$ is a bounded operator for $W^{1,p}_{c}(M)$, for any $p\geq 1$. In another words, $S(g)-S(h) \in W^{-1,p}_{loc}(M)$, for any $p\geq 1$. And for any compact set $K\subseteq M$, we have:
    \begin{align}
        \Vert S(g) - S(h) \Vert_{W^{-1,p}(K)} \leq C(n,p,K)\cdot \Vert g -h \Vert_{W^{1,p}(K)}.
    \end{align}
\end{proposition}
\begin{proof}
    From Proposition~\ref{p:scalar_distribution} and our assumption, we have, $T\in L^\infty_{loc}(M)$, $Q\in L^\infty_{loc}(M)$, $\frac{dvol(g)}{dvol(h)}\in W^{1,\infty}_{loc}(M)$. 
    
    Now for a compact set $K\subseteq M$, $\phi\in W^{1,p}_0(K)$, for some $p\geq 1$, we have $\Vert g -h\Vert_{W^{1,\infty}(K)} \leq C_\infty(K)$. In particular, for any $q\geq 1$, $\Vert g -h \Vert_{W^{1,q}(K)}\leq C_q(K)$, and it can be easily seen that $\Vert T \Vert_{L^q(K)} \leq C(n,q)\Vert g-h\Vert_{W^{1,q}(K)}$ and $\Vert Q\Vert_{L^q(K)}\leq C(n,q)\Vert g-h\Vert_{W^{1,q}(K)}$. Thus:
    \begin{align}
        |\langle S(g) - S(h), & \phi \rangle| \leq \int_M |T|\cdot |\nabla^h\phi| \cdot \Big|\frac{dvol(g)}{dvol(h)} \Big| dvol(h) + \int_M |T| \cdot |\phi|\cdot \Big|\nabla^h \frac{dvol(g)}{dvol(h)}\Big| dvol(h) \\
        &+ \int_M |Q|\cdot |\phi|\cdot \Big|\frac{dvol(g)}{dvol(h)}\Big| dvol(h)\\
        \leq & \Vert \phi \Vert_{W^{1,p}(K)}\cdot \Big(2\Vert T \Vert_{L^{p^*}(K)}\cdot \Vert\frac{dvol(g)}{dvol(h)}\Vert_{W^{1,\infty}(K)} + \Vert Q\Vert_{L^{p^*}(K)}\cdot \Vert\frac{dvol(g)}{dvol(h)}\Vert_{L^\infty(K)}\Big)\\
        \leq & C(n, p^*, K, C_\infty, C_{p^*})\cdot\Vert\phi\Vert_{W^{1,p}(K)}\cdot \Vert g - h\Vert_{W^{1,p^*}(K)}.
    \end{align}
    Hence the result follows.
\end{proof}

We will apply this to the collapsed manifold $(M_i,g_i)$ and the nearby $\mathcal{N}$-invariant metric, $\hat{g}_i$, constructed in Theorem~\ref{t.invariant_nearby_metric_estimate}, 
to obtain the following theorem.
\begin{theorem}\label{t:weak scalar curvature universal cover estimate}
    Let $\{(M_i,g_i)\}$ and $(X,g_\infty)$ satisfy the Assumption~\ref{assumption: sequence assumption}, with $k = 0$, i.e., we only assume:
    \begin{align}
        \Vert\operatorname{Rm}(g_i)\Vert_{C^0}\leq C_0.
    \end{align}
    Then, for every simply connected open subset $U\subseteq X$, on $\tilde{F_i^{-1}(U)}$, we have the estimate:
    \begin{align}
        \Vert S({g_i}) -S({\hat{g}_i})\Vert_{W^{-1,p}(K)} \leq C(n,C_0, p,  X,K)\cdot\tau_i
    \end{align}
    for every $p > n$, and for any compact set $K\subseteq \tilde{F_i^{-1}(U)}$. The integration is with respect to the volume form of $\hat{g}_i$, i.e., the $\mathcal{N}$-invariant volume form.
\end{theorem}
\begin{proof}
    Combine the estimate from Theorem~\ref{t.invariant_nearby_metric_estimate} and Proposition~\ref{p:Scalar curvature estimate negative sobolev}.
\end{proof}

\begin{remark} As observed in \cite[Section~6]{Lee_LeFloch}, the weak formulation discussed above for the scalar curvature actually holds for any curvature tensor using tensor pairings when integrating by parts. So we actually have the stronger estimate
\begin{align}
        \Vert \operatorname{Rm}({g_i}) - \operatorname{Rm}({\hat{g}_i})\Vert_{W^{-1,p}(K)} \leq C(n,C_0, p, X,K)\cdot\tau_i.
    \end{align}
    This should be useful in other applications, but for the purposes of the present paper we only need the scalar curvature estimate.
\end{remark}


\section{The limiting measure and the drift Laplacian}\label{sec:convergence_eigenvalue}
Recall the definition of measured Gromov-Hausdorff convergence: consider $\{(X_\alpha, d_{\alpha}, \mu_\alpha)\}$, where $(X_\alpha, d_{\alpha})$ is a compact metric space, and $\mu_\alpha$ is a Borel measure such that $\mu_\alpha(X_\alpha)=1$. 

\begin{definition} Let $\mu$ be a Borel measure on $X$. 
    We say:
\begin{align}
    (X_\alpha, d_\alpha, \mu_\alpha)\to (X, d_{\infty}, \mu)
\end{align}
in the \textit{measured-Gromov-Hausdorff} sense if there exists a Borel measurable map:
\begin{align}
    \psi_\alpha: X_\alpha \to X
\end{align}
and a sequence $\epsilon_\alpha\to 0$ such that:
such that:
\begin{itemize}
    \item $\psi_\alpha$ is a $\epsilon_\alpha$-Gromov-Hausdorff approximation;
    \item  The measure $(\psi_\alpha)_*\mu_\alpha \to \mu$ in the weak* topology, namely
    \begin{align}
        \lim_{\alpha\to\infty} \int_{X} u (\psi_\alpha)_*d\mu_\alpha = \int_X u d\mu
    \end{align}
    for any $u\in C^0(X)$.
\end{itemize}
\end{definition}

In our case, by assumption, take $\mu_i = \frac{dvol(g_i)}{\operatorname{vol}_{g_i}(M_i)}$, and $\psi_i = f_i$. Then we view:
\begin{align}
    (M^n_i , g_i, \mu_i)\to (X^d, g_{\infty}, \mu)
\end{align}
in the measured Gromov-Hausdorff sense.

On $X^d$, consider the sequence of functions $\{\chi_i\}$ on $X$ such that:
\begin{align}\label{eq:relative volume form}
    \chi_i :=\frac{\operatorname{vol}(F_i^{-1}(x))}{\operatorname{vol}(M_i)}.
\end{align}
The following result gives the main convergence properties of the $\chi_i$.
\begin{theorem}
\label{t:chiconv}
    Let $\{(M_i,g_i)\}$ and $(X,g_\infty)$ satisfy the Assumption~\ref{assumption: sequence assumption}, with $k = 0$.
    There exists $\chi \in C^{1,\alpha} \cap W^{2,p}$ for any $0 < \alpha < 1$ and $p > 1$, such that 
    $\chi_i \to \chi$ strongly in $C^{1,\beta}$ for any $\beta < \alpha$
    and $\chi_i \rightharpoonup  \chi$ weakly in $W^{2,p}$.
Thus, we can assume that $d \mu = \chi \operatorname{dvol}_{g_\infty}$ is satisfied pointwise.
The limiting $\chi$ is independent of the choice of fibration $F_i$, and $\chi(\hat{g}_i) = \chi({g}_i)$. 
\end{theorem}

\begin{proof}
From the construction of $F_i$ in Theorem~\ref{t.cfg} and Corollary~\ref{c.fibration_regularity}, the fibration $F_i$ has $C^{2,\alpha}$ or $W^{3,p}$ regularity. From the construction of \eqref{eq:relative volume form}, we have $\chi_i$ is $C^{1,\alpha}\cap W^{2,p}$ uniformly bounded. Then, $\chi_i \to \chi$ by Arzela-Ascoli, subsequentially, in $C^{1,\beta}$ for any $\beta <\alpha$, and $\chi_i \rightharpoonup \chi$ weakly in $W^{2,p}$.

Next, we show the limit is independent the choice of $F_i$, as long as $F_i$ is the $\tau_i$ almost Riemannian submersion in the sense of Theorem~\ref{t.cfg}. Consider the sequence $(M_i, g_i, \mu_i)$, with $\mu_i = \frac{dvol(g_i)}{\operatorname{vol}(M_i,g_i)}$. By the bounded curvature assumption, $(M_i,g_i,\mu_i)\to (X,g_\infty, \mu)$, for some measure $\mu$ on $X$ with the relation $d\mu = \chi dvol(g_\infty)$. Then for all fibration $F_i$, since they are $\tau_i$-Gromov Hausdorff approximation, we have:
\begin{align}
    \int_{F_i^{-1}(U)} d\mu_i =\frac{\operatorname{vol}(F_i^{-1}(U))}{\operatorname{vol}(M_i,g_i)} \to \int_U d\mu.
\end{align}
As argued in \cite[Lemma~3.6]{Fuk_eigen}, the relative volume function $\chi_i$ converges to some $\chi_\infty$, the measure $\chi_i dvol(g_i)$ converges to $d\mu$ in weak* topology. Thus, we have $\chi_\infty = \chi$ almost everywhere. 

Finally, for the averaged metrics $\hat{g}_i$ constructed in the proof of Theorem~\ref{t.invariant_nearby_metric_estimate}, the fiber volume is unchanged in the construction, thus
 $\chi_i(g_i) = \chi_i(\hat{g}_i)$. So the convergence is in the same sense as $\chi_i(g_i)$.
\end{proof}
In our application, we need to consider the following operator. 
\begin{definition} The drift Laplacian is 
    \begin{align}
    \Delta_\mu = \Delta_{g\infty}+\langle \nabla \log\chi, \nabla\cdot\rangle
\end{align}
    \end{definition}
    This is self-adjoint with respect to $d\mu$. To see this, note that the bilinear form on $L^2(X)$ associated with $\Delta_\mu$ is $Q(f_1,f_2) = \int_X \langle \nabla f_1, \nabla f_2 \rangle d\mu$. Formally, we have:
\begin{align}
\label{e:ibp}
\begin{split}
    \int_X\langle \nabla f_1, \nabla f_2 \rangle d\mu & = \int_X \langle \nabla f_1, \nabla f_2 \rangle \chi dvol(g_\infty)\\
    & = \int_X (-\Delta_{g_\infty} f_1\cdot\chi -\langle \nabla f_1, \nabla\chi\rangle)f_2 dvol(g_\infty)\\
    & = \int_X (-\Delta_{g_\infty} f_1 - \langle\nabla f_1, \nabla\chi\rangle/\chi)f_2 d\mu
\end{split}
\end{align}

\begin{remark}
 In  \cite{Fuk_eigen}, Fukaya proved a result about convergence of the eigenvalues of the Laplacian to the eigenvalues of the drift Laplacian on the limit. While this is closely related, we do not state it here, since we do not need it in our proof.  
 \end{remark}

 \subsection{An eigenvalue estimate for the drift Laplacian}\label{s:eignevalue estimate}

In this section, we give a natural condition on the measured limit space which implies the eigenvalue assumption in Item (2) of Theorem~\ref{cfgconj_intro}. For the measured limit space $(X^d, g_{\infty}, \mu)$, where $d\mu = \chi dvol(g_{\infty})$, 
let $\chi = e^{-f}$. 
In terms of $f$, the drift Laplacian is 
\begin{align}
\Delta_{\mu} = \Delta_{g_{\infty}} - \langle \nabla f, \nabla \cdot \rangle.
\end{align}
Next, for $n > d$, define the weighted $n$-Bakry-\'Emery-Ricci tensor by 
\begin{align}
    \operatorname{Ric}^n_f
    =
    \operatorname{Ric}_g+\nabla^2 f-\frac{1}{n-d}\,df\otimes df .
\end{align}
This was first considered by Bakry-\'Emery; see \cite{BakryEmery1985, Bakry1994}.
We have the following weighted Lichnerowicz-Obata estimate.
\begin{lemma}
\label{l:LOestimate}
Let $(X^d,g)$ be a closed Riemannian manifold, $f\in W^{2,p}(X)$ for $p > d$.
If 
\begin{align}
    \operatorname{Ric}^n_f\geq (n-1) K g
\end{align}
almost everywhere for some constant $K>0$, then 
\begin{align}
    \lambda_1(-\Delta_{\mu}) > n K .
\end{align}
\end{lemma}
\begin{proof}
This is proved in \cite{Wang2012} and the equality case is ruled out in \cite{LiWei}, since $n > d$.
Those references assumed that $f$ is smooth, however, it is easy to see that their proofs remains valid 
under our regularity assumption. 
\end{proof}
\begin{proof}[Proof of Corollary~\ref{c:lambda1}] 
We estimate
    \begin{align}
    \begin{split}
      Ric^n_f &=  Ric(g_{\infty}) + \nabla^2 f - \frac{1}{n-d} df \otimes df \\
        &\geq \Big( \frac{v}{n} + \eta \Big) g_{\infty}
        - \Vert \nabla^2 f \Vert_{C^0} g_{\infty}
        - \frac{1}{n-d} \Vert \nabla f \Vert_{C^0} g_{\infty}
        \geq \frac{v}{n} g_{\infty},
   \end{split}
    \end{align}
and the eigenvalue estimate follows from Lemma~\ref{l:LOestimate}.         
    \end{proof}

\section{The collapsing CSC metrics}
\label{section: main}

In this section, we consider the following assumption.
\begin{assumption}\label{assumption: collapsing CSC}
    Let $\{(M^n_i ,g_i)\}$ be a sequence of closed $n$-dimensional Riemannian manifolds such that:
\begin{itemize}
    \item The Riemannian curvature is uniformly bounded: \begin{align}
        \Vert\operatorname{Rm}(g_i)\Vert_{C^{0}} \leq C_0,
\end{align}
    \item The diameter is bounded from above: $d(M_i, g_i)\leq C$.
    \item The scalar curvature $S(g_i)\equiv v$, where $v$ is a constant;
    \item The sequence collapses to a lower-dimensional manifold, i.e.:
    \begin{align}
        (M^n_i,g_i,\mu_i)\to (X^d, g_\infty, \mu)
    \end{align}
    in the measured Gromov-Hausdorff sense to a manifold $X$, where $d<n$, and $d\mu = \chi dvol(g_{\infty})$.
\item If $v > 0$, the operator  $- \Delta_{\mu}$, where
\begin{align}
\label{1opassum}
\Delta_\mu h \equiv \Delta_X h  + \nabla \log \chi \cdot \nabla_Xh,
\end{align} 
does not have $\frac{v}{n-1}$ as an eigenvalue, i.e, there are no non-zero $W^{2,p}$-solutions on $X$ of 
\begin{align}
\label{e:e}
- \Delta_\mu h = \frac{v}{n-1} h.
\end{align}
\end{itemize}
\end{assumption}
The following is our main theorem. 
\begin{theorem}
\label{cfgconj}
 Let $\{(M_i^n, g_i)\}$ be a sequence of n-dimensional closed Riemannian manifolds satisfying Assumption~\ref{assumption: collapsing CSC}.
Then, for $i$ large, there exists smooth metrics $g'_i$ on $M_i$ such that:
\begin{itemize}
    \item $g_i$, $g_i'$ are close in $W^{1,p}$ sense, for any $p>n$, with $\Vert g_i-g_i' \Vert_{W^{1,p}}\to 0$ ;
    \item $g_i'$ is $\mathcal{N}$-invariant;
    \item $S(g_i')$ has constant scalar curvature satisfying $S(g_i') \equiv  v$  if $v \neq 0$, but $S(g_i') \equiv v_i$ if  $v = 0$, where $v_i \to 0$ as $i \to \infty$.
    \item $(M_i, g_i')\to (X, g_\infty)$ in the Gromov-Hausdorff sense, as $i \to \infty$.
\end{itemize}
\end{theorem}

We outline the proof of Theorem~\ref{cfgconj}. Under the Assumption~\ref{assumption: collapsing CSC}, for each $g_i$, we can find a nearby $\hat{g}_i$, which is $\mathcal{N}$-invariant and satisfies $C^{0,\alpha}$ or $W^{1,p}$ estimates, as in Theorem~\ref{t.invariant_nearby_metric_estimate}. Then, starting with these $\mathcal{N}$-invariant metrics $\hat{g}_i$, we want to conformally perturb the metric so that $(1+\phi)^{\frac{4}{n-2}}\hat{g}_i$ will satisfy the desired scalar curvature condition. In particular, we want to solve for $\mathcal{N}$-invariant CSC metric, then, $\phi$ is chosen to be an $\mathcal{N}$-invariant perturbation. Thus, the conformal equation descends to the base. See Section~\ref{s.N-invariant laplacian} for Laplacian for $\mathcal{N}$-invariant functions. The issue is that the conformal scalar equation involves the scalar curvature, which we only know in the weak sense it is close to $v$ (See Section~\ref{s.scalar distribution}). To overcome this, we view it as a map from certain Sobolev space $W^{1,p}$ to its dual, $W^{-1,p}$, see definition in Section~\ref{s.Banach spaces}. After showing uniform invertibility of the linearized operator, the theorem will then follow upon an application of the implicit function theorem.

\subsection{Laplacian on \texorpdfstring{$\mathcal{N}$}{N}-invariant functions}
\label{s.N-invariant laplacian}
For the sequence $(M_i,g_i)$ satisfying the Assumption~\ref{assumption: collapsing CSC}, for each $i$, there exists a $\Gamma\backslash\mathcal{N}$-fibration $F_i:M_i\to X$, and a $\mathcal{N}$-invariant metric $\hat{g}_i$, which is constructed in Theorem~\ref{t.invariant_nearby_metric_estimate}. Then if $v\in \ker(dF_i)_\perp$ at $p$, then $dh(v)\in \ker(dF_i)_\perp$ at $hp$, for any $h\in \mathcal{N}_L$. Thus we have a well-defined metric, $g_{X,i} = \hat{g}_i|_{\ker dF_\perp}$ on $X$, where we identify $\operatorname{im}(dF)$ with $\ker (dF)|_\perp$, with the estimate $e^{-\tau_i} g_\infty\leq g_{X,i} \leq e^{\tau_i} g_\infty$.

Let $u\in C^\infty (M)$ be a $\mathcal{N}$-invariant function. Denote the second fundamental form:
\begin{align}
    \nabla_{\partial_i}\partial_j = \bar{\nabla}_{\partial_i}\partial_j + A(\partial_i,\partial_j)
\end{align}
where $\bar{\nabla}$ is the Levi-Civita connection restricted to the fiber. Then, we have:
\begin{align}\label{eq:invariant Laplacian}
    \Delta_{g_i} u = \Delta_{g_{X,i}} u + \langle H_i, \nabla u\rangle
\end{align}
where $H_i = \operatorname{tr} A_i$, the mean curvature vector, which we can view as a vector field on $X$.

\begin{lemma}{(Lemma 3.29 in \cite{SunZhang})}\label{l.convergence invariant metrics on base}
We have $g_{X,i}\to g_\infty$ in Cheeger-Gromov $W^{1,p}$ sense on $X$, and $H_i \to \nabla_{g_\infty}\operatorname{log}\chi$ weakly in $W^{1,p}$ and strongly in $C^{0,\alpha}$. 
\end{lemma}
\begin{proof}
    For an open neighborhood $U\subseteq X$, let $u_1, ..., u_d : U\to \mathbb{R}$ be the local coordinate. Then, by the equivariant Gromov-Hausdorff convergence:
    \begin{center}
    \begin{tikzcd}
  & (\Tilde{F_j^{-1}(U)},\Tilde{g}_j, \Gamma_j,\Tilde{p}_j) \arrow[r, swap, "eq GH"] \arrow[d, "\pi_j"]
   & (\Tilde{U}_\infty,\Tilde{g}_\infty, \mathcal{N},\Tilde{p}_\infty) \arrow[d, swap, "\pi_\infty"]\\
  & (F_j^{-1}(U), g_j) \arrow[r, "F_j"]
   & (U,g_\infty)
\end{tikzcd}
\end{center}
Then $F_i\circ \pi\to \pi_\infty$, the Riemannian submersion. By Theorem~\ref{t.cfg}, $F_i$ is of uniform $W^{3,p}$ bounded. Then on the upstairs $u_\alpha\circ F_i\circ \pi_i \to u_\alpha\circ \pi_\infty$ in $W^{3,p}_{loc}$. Thus, by the $W^{1,p}$ estimate for invariant metric $g_i$ in Theorem~\ref{t.invariant_nearby_metric_estimate}, we have
\begin{align}
    \lim_{i\to\infty}\langle \nabla_{\tilde{g}_i}(u_\alpha\circ F_i\circ \pi_i),\nabla_{\tilde{g}_i}(u_\beta\circ F_i\circ \pi_i)\rangle = \langle \nabla_{\tilde{g_\infty}}(u_\alpha\circ \pi_\infty),\nabla_{\tilde{g_\infty}}(u_\beta\circ \pi_\infty)\rangle
\end{align}
descend to the base, we have the desired $W^{1,p}$ convergence on the metrics in $U$. 

For the mean curvature, by the definition of \eqref{eq:relative volume form}, we write the invariant metric $\hat{g}_i$ as 
\begin{align}
    \hat{g}_i = F_i^* g_{X,i} + (c_i\chi_i)^{\frac{2}{n-d}} g_{\Gamma\backslash\mathcal{N}}(\omega(\cdot), \omega(\cdot))
\end{align}
where $\omega$ is the connection form for the fiber bundle, $c_i = \operatorname{vol}(M_i)$ is the normalization constant, $g_{\Gamma\backslash\mathcal{N}}$ is the left invariant metric on $\Gamma\backslash \mathcal{N}$ with unit volume (See proof of Theorem~\ref{t.invariant_nearby_metric_estimate}). By Theorem~\ref{t:chiconv}, $\chi_i$ is unchanged under the averaging construction, thus the regularity follows the regularity of $\chi_i$ for the original metrics.
Now the mean curvature for the fiber is $H_i = \nabla_{g_i}\operatorname{log}\chi_i$. Since the fibration is convergent in $W^{3,p}$, we have $A_i\to A_\infty$ in $W^{1,p}$. Hence $H_i\to H_\infty = \nabla_{g_\infty}\operatorname{log}\chi$ weakly in $W^{1,p}$ 
\end{proof}
\begin{remark}
    Note that we lose one derivative in the convergence of $g_{X,i}$. For now, we do not know if the metric $g_{X,i}$ has uniformly bounded Riemannian curvature, the regularity is induced from Theorem~\ref{t.invariant_nearby_metric_estimate}. But we do not need the bounded curvature assumption on $g_{X,i}$ to carry out our analysis.  Furthermore, in the proof below we do not actually need to use the weak convergence of $H_i$, we have stated here for completeness. 
\end{remark}

\subsection{ Banach spaces for \texorpdfstring{$\mathcal{N}$}{N}-invariant functions}\label{s.Banach spaces}

As mentioned in the beginning of Section~\ref{section: main}, we need to do the analysis for the $\mathcal{N}$-invariant functions on $\tilde{F_i^{-1}(U)}$. Thus, we need to establish the corresponding Banach space for these $\mathcal{N}$-invariant objects.

We first consider the $L^p$ space. Let $U\subseteq X$ be a simply connected open set in $X$. Consider a $\mathcal{N}$-invariant function $u\in L^p_{loc}(\tilde{F_i^{-1}(U))}$. By Lemma~\ref{l.local_trivialization_regularity}, we have a local trivialization $\tilde{\Phi_{U,i}} : \tilde{F_i^{-1}(U)}\to U\times\mathcal{N}$, with uniform $C^{2,\alpha}$ and $W^{3,p}_{loc}$ estimate. Choose a open neighborhood of identity, $K\subseteq \mathcal{N}$. We may identify $K$ with $K_L\subseteq\mathcal{N}_L$, the corresponding left action. Consider the open subset:
\begin{align}
    K_i:=\tilde{\Phi_{U,i}^{-1}}(U,K)\subseteq \tilde{F_i^{-1}(U)}
\end{align}
For the $\mathcal{N}$-invariant metric, $\hat{g}_i$. Let $\tilde{g}_i$ be the lift on the local universal cover $\tilde{F_i^{-1}(U)}$. By Theorem~\ref{t.invariant_nearby_metric_estimate}, $\tilde{g}_i$ converges to the same limit $\tilde{g}_\infty$ in $C^{1,\alpha}$ Cheeger-Gromov sense. Write:
\begin{align}
    \tilde{g}_\infty = \pi_\infty^* g_\infty + \chi^{\frac{2}{n-d}}\cdot g_{\mathcal{N}}(\omega_\infty(\cdot),\omega_\infty(\cdot))
\end{align}
for some left invariant metric $g_{\mathcal{N}}$ and some connection form $\omega_\infty$. Then, 
\begin{align}
    \tilde{g}_i = \tilde{F_i^* g_i} + (c_i\chi_i)^{\frac{2}{n-d}}\cdot g_{\mathcal{N}}(\omega_i(\cdot),\omega_i(\cdot))
\end{align}
where $c_i\to 1$, $\chi_i$ is defined in \eqref{eq:relative volume form}, $\chi_i\to\chi$ in $C^{1,\alpha}$, $\omega_i\to\omega_\infty$ in $C^{1,\alpha}$.

Denote $\tilde{h}_i := (\tilde{\Phi_{U,i}^{-1}})^*\tilde{g}_i$ as the push-forward metric. Let $dvol(g_i)$ be the $\mathcal{N}$-invariant volume form given by the metric $g_i$. For a $\mathcal{N}$-invariant function $u$, denote $\Pi_i$ as the map that maps $u$ to a function on $X$, $\Pi_i (u)$. Thus, for any $x'\in F_i^{-1}(x)$, $u(x') = \Pi_i(u)(x)$. Then
\begin{align}
    \Vert u\Vert_{L^p(\bar{K_i})}^p &= \int _{K_i} |u|^p d \operatorname{vol}(\tilde{g}_i) = \int_{U\times K} |u|^p d\operatorname{vol(\tilde{h}_i)} \\
    &= \int _K c_i dvol(g_{\mathcal{N}})\cdot \int_U |\Pi_i(u)|^p \chi_i dvol(g_{X,i})
\end{align}
Now, we introduce the measure on $X$, $d\mu_i := \chi_i dvol(g_{X,i})$.
\begin{definition}
    Define the $L^p$ space $L^p(X,\mu_i)$ so that if $f\in L^p(X, \mu_i)$ then
\begin{align}
    \int_X f^p d \mu_i := \int_X f^p\chi_i dvol(g_{X,i}) < \infty. 
\end{align}
\end{definition}
Thus, we have if $\Vert u\Vert_{L^p_{loc}}$ is small, then $\Vert \Pi_i(u)\Vert_{L^p(X,\mu_i)}$ is small. Similarly, we can define the spaces $W^{k,p}(X,\mu_i)$.
\begin{definition}
    Define the $W^{k,p}(X,\mu_i)$ so that if $f\in W^{k,p}(X,\mu_i)$ then
    \begin{align}
        \sum_{i=0}^k \Vert\nabla^if\Vert_{L^p(X,\mu_i)} < \infty.
    \end{align}
\end{definition}
By Lemma~\ref{l.local_trivialization_regularity}, the local trivialization is at least $C^{2,\alpha}$, then, if $\Vert u \Vert_{W^{1,p}_{loc}}$ is small, then $\Vert \Pi_i(u) \Vert_{W^{k,p}(X,\mu_i)}$ is small.

For our application, we also want to define the negative Sobolev space $W^{-1,p}(X,\mu_i)$.

\begin{definition}
    Define the $W^{-1,p}(W,\mu_i)$ to be the dual space of $W^{1,q}(X,\mu_i)$, with $\frac{1}{p} + \frac{1}{q} = 1$. For $f\in W^{-1,p}(X,\mu_i)$,
    \begin{align}
        \Vert f\Vert_{W^{-1,p}(X,\mu_i)} = \sup \{ |\langle f, \phi\rangle| : \phi \in W^{1,q}(X,\mu_i), \Vert\phi\Vert_{W^{1,q}(X,\mu_i)}\leq 1\}.
    \end{align}
\end{definition}
Consider a $\mathcal{N}$-invariant function $f$ in $W^{-1,p}_{loc}$ as a functional, on $\tilde{F_i^{-1}(U)}$. For any $\mathcal{N}$-invariant $W^{1,q}_{loc}$-function $\phi$. Then, choose the same $K_i$ as before
\begin{align}
    \langle f,\phi\rangle &= \int_{K_i} f\cdot \phi dvol(\tilde{g_i}) = \int_K c_i dvol(g_\mathcal{N})\cdot \int_U \Pi_i(f)\cdot \Pi_i(\phi)\cdot\chi_i dvol(g_{X,i})\\
    & = \int_K c_i dvol(g_\mathcal{N})\cdot \langle \Pi_i(f), \Pi_i(\phi) \rangle.
\end{align}
Thus, if $W^{-1,p}_{loc}$ norm is small, then $W^{-1,p}(X,\mu_i)$ norm is small.

Now, for the sequence of collapsing manifolds satisfying Assumption~\ref{assumption: collapsing CSC}, we have the following estimate.
\begin{corollary}\label{c:weak sclar curvature base estimate}
    Let $\{(M_i^n ,g_i )\}$ be a sequence of manifolds satisfying Assumption~\ref{assumption: collapsing CSC}. Then for any $p > n$, we have
    \begin{align}
        \Vert \Pi_i(S({\hat{g}_i})) - v\Vert_{W^{-1,p}(X,\mu_i)} \leq C(n, C_0, p,X)\cdot \tau_i,
    \end{align}
    where $\hat{g}_i$ is the nearby $\mathcal{N}$-invariant metric constructed from $g_i$ in Theorem~\ref{t.invariant_nearby_metric_estimate}.
\end{corollary}
\begin{proof}
    By Assumption~\ref{assumption: collapsing CSC} and Theorem~\ref{t:weak scalar curvature universal cover estimate}, we have:
    \begin{align}
        \Vert S({g_i}) -v\Vert_{W^{-1,p}(K)} \leq C(n,C_0, p, X,K)\cdot\tau_i.
    \end{align}
    Note that $S({\hat{g}_i})$ and $v$ are both $\mathcal{N}$-invariant functions. Let $\phi\in W^{1,q}_c(X,\mu_i)$, consider the $\mathcal{N}$-invariant function $\tilde{\phi}$ on $K_i$, where $\Pi_i(\tilde{\phi}) = \phi$. Then, we have
    \begin{align}
        \langle \Pi_i(S({\hat{g}_i})) - v, \phi \rangle = C(K)\cdot \int_{K_i} (S(\hat{g}_i) - v)\tilde{\phi}dvol(\tilde{g}_i).
    \end{align}
    By the integration by parts formula, similar to Proposition \ref{p:Scalar curvature estimate negative sobolev}, we can estimate the integral and the boundary term. The result follows.
\end{proof}

\subsection{Conformal perturbation}\label{s:conformal perturbation}
We consider the conformal metric $\tilde{g}_{i} = \psi^{\frac{4}{n-2}} \hat{g}_{i}$,
and for the scalar curvature, we have 
\begin{align}
    -4 \frac{n-1}{n-2} \Delta_{\hat{g}_i} \psi + S(\hat{g}_i) \cdot \psi = S(\tilde{g}_i) \psi^{\frac{n+2}{n-2}}.
\end{align}
Writing $\psi = 1 + \phi$, we have 
\begin{align}
 - 4 \frac{n-1}{n-2} \Delta_{\hat{g}_i} \phi + S(\hat{g}_{i}) \cdot (1 + \phi) =  S(\tilde{g}_{i}) (1 + \phi)^{\frac{n+2}{n-2}},
\end{align}
which we write as 
\begin{align}
 S(\tilde{g}_{i}) = (1 + \phi)^{- \frac{n+2}{n-2}} \Big( - 4 \frac{n-1}{n-2} \Delta_{\hat{g}_{i}} \phi + S(\hat{g}_{i}) \cdot (1 + \phi)  \Big)
\end{align}

Since our goal is to construct an $\mathcal{N}$-invariant CSC metric, we want to conformally perturb the metric in an $\mathcal{N}$-invariant way. 
For this, we define a mapping $P_{\hat{g}_i}$ by
\begin{align}
    \phi\mapsto  - 4 \frac{n-1}{n-2} \Delta_{\hat{g}_i} \phi + S(\hat{g}_{i}) \cdot (1 + \phi)    -  v\cdot (1 + \phi)^{ \frac{n+2}{n-2}}. 
\end{align}
We denote $\Theta_i: = v -S(\hat{g}_i)$. Then we can write $P$ as
\begin{align}
    \phi \mapsto - 4 \frac{n-1}{n-2} \Delta_{\hat{g}_i} \phi - (v\cdot\frac{4}{n-2} - \Theta_i)\cdot \phi - \Theta_i - v\cdot \Big((1+\phi)^{\frac{n+2}{n-2}} - 1 - \frac{n+2}{n-2}\cdot \phi \Big).
\end{align}

Write
\begin{align} \label{1Pexp}
P_g(h) = P_g(0) + L_gh + Q_g(h),
\end{align}
where 
\begin{align}
    &P_g(0):= \Theta\\
    &L_g : = - 4 \frac{n-1}{n-2} \Delta_{g}  - (v\cdot\frac{4}{n-2} - \Theta)\\
    &Q_g(h) := - v\cdot \Big((1+h)^{\frac{n+2}{n-2}} - 1 - \frac{n+2}{n-2}\cdot h \Big). 
\end{align}

When $\phi\in C^\infty(M_i)$ is an $\mathcal{N}$-invariant function, then it is easy to see $P_{\hat{g}_i}(\phi)$ is an $\mathcal{N}$-invariant function. By \eqref{eq:invariant Laplacian}, we have
\begin{align}
    \Delta_{\hat{g}_i} \phi = \Delta_{g_{X,i}} \Pi_i(\phi)+ \langle H_i, \nabla \Pi_i(\phi)\rangle = \Delta_{g_{X,i}} \Pi_i(\phi) + \langle \nabla\operatorname{log}\chi_i, \nabla \Pi_i(\phi)\rangle,
\end{align}
where $\Pi_i(\phi)\in C^\infty(X)$. Denote $\Delta_{\mu_i} = \Delta_{g_{X.i}} + \langle \nabla \operatorname{log}\chi_i, \cdot\rangle$. Define:
\begin{align}
    \hat{L}_{\hat{g}_i} : = - 4 \frac{n-1}{n-2} \Delta_{\mu_i}  - (v\cdot\frac{4}{n-2} - \Theta_i).
\end{align}
Replace $L_{\hat{g}_i}$ with $\hat{L}_{\hat{g}_i}$, we view $P_{\hat{g}_i}: C^\infty(X)\to C^\infty(X)$. In particular, We have the following mapping property of $P_g$. 
\begin{proposition}
    The mapping $P_{\hat{g}_i}$ maps $W^{1,p}(X,\mu_i)$ to  $W^{-1,p}(X,\mu_i)$, for any $p> d$, i.e.,
    \begin{align}
        P_{\hat{g}_i} : W^{1,p}(X,\mu_i) \to W^{-1,p}(X,\mu_i).
    \end{align}
\end{proposition}
\begin{proof}
    Similar to the discussion in Subsection~\ref{sec:convergence_eigenvalue}, $\Delta_{\mu_i}  $ is a self-adjoint operator with respect to the normalized volume form $d\mu_i = \chi_i dvol(g_i)$. Thus, for $h\in W^{1,p}(X,\mu_i)$, for any $\psi\in W^{1,p^*}(X,\mu_i)$, with $\frac{1}{p} + \frac{1}{p^*} = 1$, we have $\langle\Delta_{\mu_i} h, \psi\rangle = \int_X -\langle\nabla h, \nabla \psi\rangle d\mu_i $, which defines a bounded functional. Since $p>d$, by Sobolev embedding, $\Vert h \Vert_{C^0}\leq C(X,g_{X,i})\Vert h \Vert_{W^{1,p}(X,\mu_i)}$, then the result follows.
\end{proof}

The following proposition regarding the structure of the
operator $P_g$ will be key in order to apply the implicit function theorem later. 
\begin{proposition} \label{1quadest} 
We have the following properties for $P_{g_i}$:  

\noindent $(i)$ For $i \gg 0$, there exists a constant $C$ so that for $p > n$
\begin{align}
\Vert P_{\hat{g}_i}(0) \Vert_{W^{-1,p}(X,\mu_i)} \leq C \tau_i. 
\end{align}

\noindent $(ii)$ The linearized operator of $P_{\hat{g}_i}$ at $h=0$ is given by $\hat{L}_{\hat{g}_i}$.

\noindent $(iii)$  
For $h_j\in W^{1,p}(X,\mu_i)$ with
$\|h_j \|_{W^{1,p}(X,\mu_i)} < s_0$ sufficiently small, $j =1,2$.
There exists a constant $C_2 = C_2(s_0)$ so that $Q_g$ satisfies the following estimate:
\begin{align} \label{1quadstructure}
\Vert Q_{\hat{g}_i}(h_1) -  Q_{\hat{g}_i}(h_2)\Vert_{W^{-1,p}(X,\mu_i)} \leq
C_2(  \Vert h_1 \Vert_{W^{1,p}(X,\mu_i)}
+ \Vert h_2 \Vert_{W^{1,p}(X,\mu_i)}) \cdot
\Vert  h_1 - h_2 \Vert_{W^{1,p}(X,\mu_i)}.
\end{align}
\end{proposition}
\begin{proof}
    The proof of (i) follows Corollary~\ref{c:weak sclar curvature base estimate}. 
    \begin{align}
        \Vert P_{g_i}(0)\Vert_{W^{-1,p}(X,\mu_i)} = \Vert \Theta_i \Vert_{W^{-1,p}(X,\mu_i)} \leq C\tau_i.
    \end{align}
The proof of (ii) is by direct computation. The estimate in (iii) easily follows from Taylor's theorem since the nonlinear term is an analytic function of $h$, for $h$ sufficiently small. \end{proof}
We also have the following Calder\'on-Zygmund estimate on the base with respect to the $d\mu_i$ measure.
\begin{proposition}\label{p:elliptic estimate} For $p>d$, and $i \gg 1$, there exists a constant $C$, independent of $i$ such that 
\begin{align}
\label{1wse}
\Vert h \Vert_{W^{1,p}(X,\mu_i)} \leq C \Big(   \Vert \hat{L}_{\hat{g}_i} h \Vert_{W^{-1,p}(X,\mu_i)}
+ \Vert h \Vert_{W^{-1,p}(X,\mu_i)} \Big).
\end{align}
\end{proposition}
\begin{proof}
From Lemma~\ref{l.convergence invariant metrics on base}), we know that $g_{X,i}$ converges to $g_\infty$ on $X$ in the $C^{0,\alpha}$ or $W^{1,p}$ sense. Since they are non-collapsing metrics on $X$, it is enough to apply standard Calder\'on-Zygmund estimate for divergence-form operators, and we obtain \eqref{1wse} with $C$ independent of $i$, for $i$ sufficiently large.
\end{proof}

To apply the inverse function theorem, we need a bounded inverse theorem for the linear operator $\hat{L}_g$. In the case $v \neq 0$ we have the following.
\begin{proposition}
\label{p:vneq0}
   Let $\{(M_i,g_i)\}$ satisfies Assumption~\ref{assumption: collapsing CSC} with limiting manifold $(X,g_\infty)$. If $v \neq 0$, then for $i$ sufficiently large the map $\hat{L}_{\hat{g}_i} : W^{1,p}(X,\mu_i) \rightarrow W^{-1,p}(X,\mu_i)$ is uniformly injective for $p>n$: i.e., there is a constant $\delta_0 > 0$
which is independent of $i$ such that
\begin{align} \label{1sbelow}
\| \hat{L}_{\hat{g}_i} h \|_{W^{-1,p}(X,\mu_i)} \geq \delta_0 \big( \| h \|_{W^{1,p}(X,\mu_i)} \big),
\end{align}
for all $h \in W^{1,p}(X,\mu_i)$, with $p>n$.
\end{proposition}
\begin{proof}
    We argue via contradiction: if \eqref{1sbelow} does not hold, there is a subsequence $h_i\in W^{1,p}(X,\mu_i)$ with
    \begin{align} 
\label{1assumption1} 
 \| h_i \|_{W^{1,p}(X,\mu_i)} &= 1 \ \ \forall i, \\
\label{1assumption2}
\epsilon_i \equiv \hat{L}_{\hat{g}_i} h_i  & \rightarrow 0 \ \ \mbox{in } W^{-1,p}(X,\mu_i).
 \end{align}
Since $p>n>d$, there exists $h_\infty\in C^{0,\alpha}(X)$, where $\alpha=1-\frac{d}{p}$, so that $h_i\to h_\infty$ in $C^{0,\beta}$ sense, for any $\beta <\alpha$. 
Consider $p^*$ so that $\frac{1}{p}+\frac{1}{p^*}=1$. 
Since $d\mu_i$ converges to $d\mu$ in weak* topology (cf. Section~\ref{sec:convergence_eigenvalue}). Then, for any vector field $V$ with bounded $L^{p^*}(X,g_\infty)$ norm, then it has bounded $L^{p^*}(X,\mu_i)$ for all sufficiently large $i$. Thus, 
\begin{align}
    \int_X \langle\nabla h_i,V\rangle\cdot d\mu_i\to \int_X \langle\nabla h_\infty, V\rangle\cdot d\mu.
\end{align}
Noting that for any $W^{1,p^*}(X, \mu)$ function $\psi$, then $\psi \in W^{1,p^*}(X,\mu_i)$, for sufficiently large $i$, then since $\Delta_{\mu_i}$ is self-adjoint with respect to $d\mu_i$, we have
\begin{align}
    \langle \hat{L}_{\hat{g}_i}h_i, \psi\rangle = \int_X \Big( 4\frac{n-1}{n-2}\langle \nabla h_i,\nabla \psi \rangle - \frac{4}{n-2}(v\cdot \psi)h_i \Big) d\mu_i + 
\int_X \Theta_i(h_i\cdot\psi) d\mu_i = \epsilon_i(\psi). 
\end{align}
Therefore, when $i \to \infty$, by Corollary~\ref{c:weak sclar curvature base estimate}, we have a weak solution
\begin{align}
    \int_X \Big( \langle \nabla h_\infty,\nabla \psi \rangle - \frac{v}{n-1} h_\infty \psi \Big) d\mu = 0.
\end{align}
By elliptic regularity $h_{\infty} \in W^{2,p}(X,\mu)$ and is a strong solution of \eqref{e:e}, so by Assumption~\ref{assumption: collapsing CSC}, we have $h_\infty\equiv 0$. By the elliptic estimate from Proposition~\ref{p:elliptic estimate}, $h_i \to h_{\infty}$ strongly in $W^{1,p}$, which contradicts \eqref{1assumption1}.
\end{proof}

In the case $v = 0$, we need to consider the space $W_0^{1,p}(X,\mu_i)\subseteq W^{1,p}(X,\mu_i)$ defined as
\begin{align}
    \int_X \phi d\mu_i = 0
\end{align}
for $\phi \in W^{1,p}_0(X,\mu_i)$.  We also define 
\begin{align}
W^{-1,p}_0 (X, \mu_i) \equiv \{ f \in W^{-1,p}(X, \mu_i) \ | \  f(1) = 0 \}.
\end{align}
We define the nonlinear functional:
\begin{align}
    v_i = v_i(\phi) = \frac{\int_X S(\hat{g}_i) (1+\phi)d\mu_i}{\int_X (1+\phi)^{\frac{n+2}{n-2}}d\mu_i}.
\end{align}
Thus
\begin{align}
    v_i(0) = \int_X S(\hat{g}_i) d\mu_i
\end{align}
Then define 
\begin{align}P_{0,\hat{g}_i}: W_0^{1,p}(X,\mu) \rightarrow  W^{-1,p}_0 (X, \mu)
\end{align}
by 
\begin{align}
         \phi\mapsto  - 4 \frac{n-1}{n-2} \Delta_{\hat{g}_i} \phi + S(\hat{g}_{i}) \cdot (1 + \phi)    -  v_i(\phi)\cdot (1 + \phi)^{ \frac{n+2}{n-2}},
\end{align}
The linearized operator can be computed as
\begin{align}
    \hat{L}_{0,\hat{g}_i}h = - 4 \frac{n-1}{n-2} \Delta_{\mu_i}h  + \Big(\Theta_i- \frac{n+2}{n-2} v_i(0)\Big)h-  \int_X \Big(\Theta_i - \frac{n+2}{n-2} v_i(0) \Big)h d\mu_i'.
\end{align}
We also have
\begin{align}
    Q_{0,\hat{g}_i}(h) = P_{0,\hat{g}_i}(h) - P_{0,\hat{g}_i}(0) -\hat{L}_{0,\hat{g}_i}h,
\end{align}
and it can be computed as
\begin{align}
    Q_{0,\hat{g}_i}(h) = -v_i(h)\cdot (1+h)^{\frac{n+2}{n-2}} + \big(\frac{n+2}{n-2}\big)v_i(0)\cdot h + \int_X \Big(\Theta_i - \big(\frac{n+2}{n-2}\big)v_i(0)\Big)h d\mu_i + v_i(0).
\end{align}
The analog of Proposition~\ref{1quadest} holds for $P_{0,\hat{g}_i}$, we do not restate it here. Moreover, we have the following.
\begin{proposition}
\label{p:veq0}
   Assume that $\{(M_i,g_i)\}$ satisfies Assumption~\ref{assumption: collapsing CSC} with limiting manifold $(X,g_\infty)$ with $v = 0$. Then for $i$ sufficiently large the map $\hat{L}_{0,\hat{g}_i} : W^{1,p}_0(X,\mu_i) \rightarrow W^{-1,p}_0(X,\mu_i)$ is uniformly injective for $p>n$: i.e., there is a constant $\delta_0 > 0$
which is independent of $i$ such that
\begin{align} \label{1sbelow0}
\| \hat{L}_{0,\hat{g}_i} h \|_{W^{-1,p}(X,\mu_i)} \geq \delta_0 \big( \| h \|_{W^{1,p}(X,\mu_i)} \big),
\end{align}
for all $h \in W^{1,p}_0(X,\mu_i)$, with $p>n$.
In particular, we have $\Ker(\hat{L}_{\hat{g}_i}) \simeq \Coker ( \hat{L}_{\hat{g}_i}) \simeq \RR$ is one-dimensional.
\end{proposition}
\begin{proof} The proof is almost identical to the proof of Proposition~\ref{p:vneq0}. The only difference is to show that the limit $h_{\infty}$ of the $h_i$ in the contradiction argument has mean value zero. This can be easily seen from the fact that $\Vert h_i\Vert_{L^\infty}$ is uniformly bounded (by Sobolev embedding), the fact that $\chi_i$ converges to $\chi_\infty$ in $C^{1,\alpha}$ and the dominated convergence theorem. Thus, the limit $h_\infty$ is in $W^{1,p}_0(X,\mu)$. Similar to the proof of Proposition~\ref{p:vneq0}, we have $\Delta_{\mu} h_\infty = 0$ weakly, and then strongly by elliptic regularity. From \eqref{e:ibp}, we conclude that $h_{\infty}$ is constant, and therefore $h_{\infty} \equiv 0$.
By the elliptic estimate from Proposition~\ref{p:elliptic estimate}, $h_i \to h_{\infty}$ strongly in $W^{1,p}$, which is a contradiction.
\end{proof}

We next quote without proof the following standard implicit
function theorem:
\begin{lemma}
\label{ift}
Let $H : E \rightarrow F$ be a smooth map between Banach spaces.
Define $Q = H - H(0) - H'(0)$. Assume that there are positive constants
$C_1, s_0, C_2$ so that the following are satisfied:
\begin{itemize}
\item $(1)$ The nonlinear term $Q$ satisfies
\begin{align}
\label{iquad}
\Vert Q(x) - Q(y) \Vert_F \leq C_1 (\Vert x \Vert_E + \Vert y \Vert_E)
\Vert x - y \Vert_E
\end{align}
for every $x, y \in B_E (0, s_0)$.
\item $(2)$ The linearized operator at $0$, $H'(0) : E \rightarrow F$
is an isomorphism with inverse bounded by $C_2$.
\end{itemize}
If
\begin{align}
s < \min \Big( s_0, \frac{1}{2 C_1 C_2} \Big)
\end{align}
and
\begin{align}
\Vert H(0) \Vert_F < \frac{ s}{2 C_2},
\end{align}
Then there is a unique solution $x \in B_E(0,s)$ of the
equation $H(x) = 0$.
\end{lemma}

We end this section with the following existence theorem:
\begin{theorem}\label{t.solution with estimate} Assume that $\{(M_i,g_i)\}$ satisfies Assumption~\ref{assumption: collapsing CSC} with limiting manifold $(X,g_\infty)$ with $v = 0$. For all $i \gg 0$,
there exists $h_i \in C^{\infty}(X)$ satisfying
\begin{align}
\label{thetanorm}
\Vert h_i \Vert_{W^{1,p}(X,\mu_i)} < C \tau_i,
\end{align}
where $p > n$ and $\tau_i \to 0$ as $i \to \infty$, solving
\begin{align}
\label{bigeqn}
P_{\hat{g}_i}(h_i) = 0.
\end{align}
if $v \neq 0$ and 
\begin{align}
\label{bigeqn0}
P_{0,\hat{g}_i}(h_i) = 0
\end{align}
if $v = 0$
\end{theorem}
\begin{proof} This follows from Lemma~\ref{ift}, using Proposition~\ref{1quadest} and Proposition~\ref{p:vneq0} in the case $v \neq 0$,
and Proposition~\ref{p:veq0} in the case $v = 0$. Smoothness of $h_i$ follows from standard elliptic regularity. 
\end{proof}
To finish the proof of Theorem~\ref{cfgconj}, we show that  the sequence $(M_i,g_i'= (1+h_i)^{\frac{4}{n-2}} \hat{g}_i)$ satisfies
all of the conditions stated there.

\begin{proof}[Proof of Theorem~\ref{cfgconj}]
    By the Assumption~\ref{assumption: collapsing CSC}, there exists an $\mathcal{N}$-invariant nearby metric $\hat{g}_i$ from Theorem~\ref{t.invariant_nearby_metric_estimate}. View $\hat{g}_i$ as the background metric, by Theorem~\ref{t.solution with estimate}, we can find a $\mathcal{N}$-invariant perturbation $h_i$ so that $g_i':=(1+h_i)^{\frac{4}{n-2}}\hat{g}_i$ is an $\mathcal{N}$-invariant metric satisfying the following:
    \begin{enumerate}
        \item[(1)] In the case when $v<0$, there is no eigenvalue assumption. By Theorem~\ref{t.solution with estimate}, then $S(g_i')\equiv v$;
        \item[(2)] In the case when $v = 0$, there is no eigenvalue assumption. By Theorem~\ref{t.solution with estimate}, $S(g'_i) = v_i$, for some constant $v_i \to 0$ as $i\to\infty$;
        \item[(3)] In the case when $v >0$, by the eigenvalue assumption, we have a solution $h_i$ by Theorem~\ref{t.solution with estimate}, and $S(g_i')\equiv v$.
        
    \end{enumerate}
    By the estimate \eqref{thetanorm} and Sobolev embedding, the new sequence $\{(M_i,g_i')\}$ converges to the same Gromov-Hausdorff limit.  
\end{proof}


\appendix

\section{Proof of Theorem~\ref{t.cfg}}

    We will first construct such $C^{2,\alpha}$ fibration map under the assumption, following the idea of the proof of Proposition 6.6 of \cite{NaberZhang}.

    \textbf{Step 1}: the existence of local fibration map.

    We first fix a scale $\epsilon >0$ and consider the rescaled metric $g'_i := \epsilon^{-2}g_i$ converge to $g'_\infty = \epsilon^{-2}g_\infty$. By the assumption, for any sequence $\{x_i\}$ with $x_i\in M_i^n$ and $x_i\to x\in X^d$ and the ball $(B_4(x_i),g'_i) \to (B_4(x),g'_\infty)$ in the Gromov-Hausdorff sense. We can choose the scaling parameter $\epsilon$ so small such that
    \begin{enumerate}
        \item[(1)] The Gromov-Hausdorff distance $(B_4(x),g'_\infty)$ with the Euclidean ball $B_4(0^d)$ is less than $\epsilon$;
        \item[(2)] The Ricci curvature 
        $\operatorname{Ric}(g_i)\geq -(n-1)\epsilon$.
        \item [(3)] The harmonic radius of $g_\infty'$ is greater than $3$.
    \end{enumerate}
    Then, we can use the construction of harmonic functions from \cite[Section~6]{CheegerColdingwraped}. For $i$ sufficiently large, there exists a harmonic map
    \begin{align}
        \Phi_{x_i}^i:B_2(x_i)\to \mathbb{R}^d
    \end{align}
    such that $\Phi_{x_i}^i:B_2(x_i)\to B_2(0^d) $ is a $\tau(\epsilon)$-Gromov-Hausdorff approximation with the following gradient estimate
    \begin{align}
        |\nabla \Phi_{x_i}^{i,\alpha}|(p)\leq 1+\tau(\epsilon),
    \end{align}
    and the integral estimate
    \begin{align}
        \frac{1}{\operatorname{vol}(B_2(x_i))}\Big(\sum_{\alpha,\beta = 1}^d\int_{B_2(x_i)} \Big( |\langle\nabla\Phi_{x_i}^{i,\alpha},\nabla\Phi_{x_i}^{i,\beta}\rangle - \delta_{\alpha,\beta}|   + \frac{1}{d} |\operatorname{Hess}\Phi_{x_i}^{i,\alpha}|^2   \Big) dvol(g_i)\Big) < \tau(\epsilon).
\end{align}

    To get pointwise $C^{2,\alpha}$ estimate for $\Phi_{x_i}^i$, consider the pull-back metric $(B_2(0^n),\hat{g}'_i)$, where we define $\hat{g}'_i := \operatorname{exp}_{x_i}^*g'_i$ as the pull-back metric on the local universal cover $B_2(0^n)\subseteq \mathbb{R}^n$. Then, the pull-back function $\hat{\Phi}_{x_i}^i := \Phi_{x_i}^i\circ\operatorname{exp}_{x_i}$ is harmonic. By our curvature assumption, the conjugate radius is bounded from below and thus the injectivity radius of $(B_2(0^n),\hat{g}'_i)$ is bounded from below. Thus, we have the metric, under a suitable coordinate, is uniformly $C^{1,\alpha}$ bounded. Thus, the coefficients of the Laplacian $\Delta_{\hat{g}'_i}$ are uniformly $C^{0,\alpha}$ bounded. Thus, we can apply the standard Schauder estimate and conclude the pointwise $C^{2,\alpha}$ estimate
    \begin{align}\label{e:harmonic_coord_C2alpha}
        \Vert \Phi_{x_i}^{i,\alpha}\Vert_{C^{2,\alpha}(B_1(x_i))} < C(n).
    \end{align}

    To check the nondegeneracy, we need to show that there is no critical point in $B_1(x_i)$. In fact, we want to show the following pointwise estimate:
    \begin{align}
        \sum_{\alpha,\beta = 1}^d|\langle\nabla\Phi_{x_i}^{i,\alpha},\nabla\Phi_{x_i}^{i,\beta}\rangle - \delta_{\alpha,\beta}| < \tau(\epsilon).
    \end{align}
    This can be proved by combining uniform relative volume comparison and a compactness argument.

    On the limit $X$. By our assumption, $X$ is a compact Riemannian manifold. Let $B_3(x_i)$ converge to $B_3(x_\infty)\subseteq X$. Then, there exists harmonic coordinate system
    \begin{align}
        H_{x_\infty} = (u_1, ..u_d):B_3(x_\infty)\to \mathbb{R}^d
    \end{align}
    with the point-wise estimate $\sum_{\alpha,\beta=1}^d |\langle \nabla u_\alpha,\nabla u_\beta\rangle - \delta_{\alpha,\beta}|<\tau(\epsilon)$, and uniform $C^{2,\alpha}$ estimate. Thus, we can define a local  map $F_{x_i} = H_{x_\infty}^{-1}\circ \Phi^{i}_{x_i}: B_1(x_i)\to B_1(x_\infty)$, with point-wise $C^{2,\alpha}$ estimate. By choosing $\epsilon$ small enough, there is no critical point and thus defines a local fiber bundle with fiber of dimension $n-d$.

    \textbf{Step 2}: get a global fibration map via patching.

    Again, we fix the scaling $\epsilon$ as well. Now, we want to define a partition of unity for each $M_i$ when $i$ is large with uniform $C^{2,\alpha}$ estimate. First, we need to construct a set of cut-off functions centered at any $p_i\in M_i$ with the cutting radius 1. In other words, we want to find $\psi_{p_i}$ so that:
    \begin{align}
        \psi_{p_i}(x) = \left\{
        \begin{aligned}
            &1, \quad x\in B_{\frac{1}{2}}(p_i)\\
            &0, \quad x\in M_i\setminus B_1(p_i)
        \end{aligned}
        \right.
    \end{align}
    for every $p_i$ with $\Vert \psi_{p_i}\Vert_{C^{2,\alpha}} \leq C$, which is independent of $i$. This cut-off function is constructed in Theorem 6.33 of \cite{CheegerColdingwraped}. 
    First, we construct the following functions:
    \begin{align}
        \left\{
        \begin{aligned}
            &\Delta k_{p_i} = \kappa, \quad x \in B_1(p_i)\setminus B_{\frac{1}{2}}(p_i)\\
            & k_{p_i} = 1,\quad x\in \partial B_{\frac{1}{2}}(p_i)\\
            & k_{p_i} = 0, \quad x\in \partial B_1(p_i)
        \end{aligned}
        \right. ,
    \end{align}
    for some $\kappa>0$, which only depends on the dimension. Take $\Psi:\mathbb{R}\to [0,1]$ as a smooth cut-off function. Then, we construct $\psi_{p_i} = \Psi\circ k_{p_i}$. It is the desired cut-off function by \cite{CheegerColdingwraped}. For the regularity, by our curvature boundedness assumption, the coefficients of the Laplacian are uniformly $C^{0,\alpha}$ bounded, so by Schauder estimate, we can improve the regularity of such cut-off function to $C^{2,\alpha}$, following the argument of \cite{CheegerColdingwraped}.

    We can construct the partition of unity with uniform $C^{2,\alpha}$ estimate now. We choose a cover $\{B_1(q^k)\}_{k=1}^N$ so that $M_i\subseteq \cup_{k=1}^N B_1(q^k)$ and we define the partition of unity subordinate to this cover by:
    \begin{align}
        \phi_k = \frac{\psi_{q^k}}{\sum_{j=1}^N \psi_{q^j}}.
    \end{align}
    Note that $N$ is uniformly bounded by the assumption of uniformly bounded curvature of the sequence and a lower injectivity radius bound on $X$.

    To patch local fibration maps together, we need to introduce the following function on $X$ to determine the image of the fibration. Define the distance-like function:
    \begin{align}
        \eta_{q_\infty}:B_1(q_\infty)\times B_1(q_\infty)\to \mathbb{R}^+, \quad (x,y)\mapsto \sum_{\alpha = 1}^d |u_\alpha(x)-u_\alpha(y)|^2
    \end{align}
    and the global distance-like function $\eta : M_i\times X \to \mathbb{R}^+$:
    \begin{align}
        \eta(p_i, p_\infty) = \sum_{k = 1}^N \phi_k\eta_{q_\infty^k}(F_{q_k}(p_i),p_\infty).
    \end{align}
    On each local chart $H_{p_\infty^k}$, by the harmonic coordinate estimate, $|\eta_{q_\infty} - d_{g_\infty'}^2|\leq \tau(\epsilon)$. Thus, it can be easily seen that $\eta(p_i,\cdot)$ is a convex function on $X$, hence has a unique minimum, $\rho(p_i)$. Thus, we can define the global map:
    \begin{align}
        F_i : M_i \to X,\quad p_i \mapsto \rho(p_i).
    \end{align}
    Now, we examine the regularity. Note $\eta_{q_\infty}$ is smooth, then $\eta(p_i,\cdot)$ is a smooth function, $C^{2,\alpha}$ in $p_i$ variable.
    Consider $G(p_i,y) = \nabla_y \eta (p_i,y)$. By convexity, $\nabla_y G(p_i,y) > 0$. Thus, by the implicit function theorem, the map $p_i\mapsto \rho(p_i)$ is $C^{2,\alpha}$.
    It is a well-defined map, with $d_{g_\infty'}(F_i(p_i), F_{x_i}(p_i))\leq \tau(\epsilon)$ and 
    \begin{align}
        |\langle\nabla \Phi_{x_i}^{i,\alpha}, \nabla\Phi_{x_i}^{i,\beta}\rangle - \langle\nabla (F_i\circ H_{x_\infty})^{\alpha}, \nabla (F_i\circ H_{x_\infty})^{\beta}\rangle|\leq \tau(\epsilon).
    \end{align}
    Similarly, $F_i$ is a $C^{2,\alpha}$ fibration map with fiber of dimension $n-d$.

    Now, we shall prove the rest of the theorem. To show the $C^{0,\alpha}$ bound for the fiber second fundamental form. Note that we have:
    \begin{align}
        A_{F_i^{-1}(x)}(V,W) = \frac{\operatorname{Hess}F_i(V,W)}{|\nabla F_i|}
    \end{align}
    for $V,W\in\operatorname{ker}(dF_i)$. Thus, we have:
    \begin{align}
        \Vert A_{F_i^{-1}(x)}\Vert_{C^{0,\alpha}} \leq C(n)\Vert F_i\Vert_{C^{2,\alpha}}.
    \end{align}

    To show that $F_i$ is an almost-Riemannian submersion, we can use the contradiction argument given in \cite{SunZhang}. Suppose not, there exists a sequence of vectors $v_i\in T_{x_i} M_i$, perpendicular to the fiber direction, and a constant $\tau_0\geq 0$ so that $|\frac{|dF_i(v_i)|_{g_\infty}}{|v_i|_{g_i}}-1|\geq \tau_0$. We may assume $|v_i|_{g_i} = 1$. Consider the local equivariant Gromov-Hausdorff convergence of local universal cover
    \begin{center}
    \begin{tikzcd}
    & (\tilde{B}_{\tilde{x}_i}(r), \tilde{g}_i, G_i) \arrow[r,  "C^{1,\alpha}"] \arrow[d, "\pi_i"]
    & (\tilde{B}_{\tilde{x}_\infty}(r),\Tilde{g}_\infty, G_\infty) \arrow[d, swap, "\pi_\infty"]\\
    & (B_{x_i}(r), g_i) \arrow[r, "\text{GH}"]
    & (B_{x_\infty}(r), g_\infty)
    \end{tikzcd}
    \end{center}
    where $\pi_i$ is the local universal cover map, $\pi_\infty$ is a Riemannian submersion. Thus, $\tilde{F}_i = F_i\circ \pi_i$ converges to $\pi_\infty$ in $C^{2,\beta}$. Let $\tilde{v}_i$ be a lift of $v_i$. Thus, $\tilde{v}_i$ converges to $\tilde{v}_\infty$ with $|d\pi_\infty \tilde{v}_\infty| = 1$. This violates the assumption $|\frac{|dF_i(v_i)|_{g_\infty}}{|v_i|_{g_i}}-1|\geq \tau_0$.


\bibliographystyle{amsalpha}
\bibliography{Collapsing}

\end{document}